\documentclass[11pt]{amsart}
\usepackage{enumitem}
\usepackage{color}
\usepackage{amssymb,verbatim}
\usepackage{mathrsfs}
\usepackage{dsfont}
\usepackage[colorlinks,citecolor=blue,urlcolor=black,linkcolor=black]{hyperref}
\usepackage[mathcal]{eucal}

\newtheorem*{MAIN}{Main Theorem}
\newtheorem*{conjecture*}{Conjecture}
\newtheorem{theorem}{Theorem}[section]
\newtheorem*{theorem*}{Theorem}
\newtheorem*{definition*}{Definition}
\newtheorem{prop}[theorem]{Proposition}
\newtheorem{conjecture}[theorem]{Conjecture}
\newtheorem{claim}{Claim}[theorem]

\newtheorem{lemma}[theorem]{Lemma}
\newtheorem{cor}[theorem]{Corollary}

\newtheorem{question}{Question}
\newtheorem*{question*}{Question}
\newtheorem{fact}[theorem]{Fact}

\makeatother

\theoremstyle{definition}
\newtheorem{definition}[theorem]{Definition}
\newtheorem{notation}[theorem]{Notation}

\newtheorem*{IH}{Induction Hypothesis}

\theoremstyle{remark}
\newtheorem{remark}[theorem]{Remark}

\def\s{\subseteq}

\def\forces{\Vdash}
\def\br{\blacktriangleright}

\newcommand{\one}{\mathop{1\hskip-3pt {\rm l}}} 
\newcommand{\Ult}{\mathrm{Ult}}
\renewcommand{\restriction}{\mathbin\upharpoonright}    	
\renewcommand{\mid}{\mathrel{|}\allowbreak}

\newcommand{\Col}{\mathop{\mathrm{Col}}}
\newcommand{\Cn}{\mathop{C^{(n)}}}

\newcommand\old[1]{}

\DeclareMathOperator{\rank}{rank}

\DeclareMathOperator{\supp}{supp}

\DeclareMathOperator{\crit}{crit}

\DeclareMathOperator{\dom}{dom}
\DeclareMathOperator{\ran}{ran}

\DeclareMathOperator{\otp}{otp}

\DeclareMathOperator{\acc}{acc}

\DeclareMathOperator{\cf}{cf}

\DeclareMathOperator{\ord}{Ord}

\DeclareMathOperator{\hod}{HOD}
\DeclareMathOperator{\id}{id}

\newcommand{\ZFC}{\mathrm{ZFC}}

\newcommand{\Ord}{\mathrm{Ord}}

\newcommand\ale[1]{\marginpar{Alejandro: #1}}

\title[Axiom $\mathcal{A}$ and supercompactness]{Axiom $\mathcal{A}$ and supercompactness}

\thanks{The author acknowledges support from the Department of Mathematics at Harvard University as well as from the Harvard Center of Mathematical Sciences and Applications. This version is from June 6, 2024. Updates may be posted in the authors's webpage \url{www.alejandropovedaruzafa.com}}

\author[Poveda]{Alejandro Poveda}
\address{Harvard University, Department of Mathematics and Center of Mathematical Sciences and Applications, Cambridge (MA), 02138, USA}
\email{alejandro@cmsa.fas.harvard.edu}
\urladdr{https://scholar.harvard.edu/apoveda/home}
\urladdr{www.alejandropovedaruzafa.com}

\begin{document}

\begin{abstract}
   We produce a model where every supercompact cardinal is $C^{(1)}$-supercompact with inaccessible targets. This is a significant improvement of the 
   main identity-crises configuration obtained in \cite{HMP} and provides a definitive answer to a question of Bagaria \cite[p.19]{Bag}. This configuration is a consequence of a new axiom we introduce --called  $\mathcal{A}$-- which is showed to be compatible with Woodin's $I_0$ cardinals.  
  We also   answer a question of V. Gitman and G. Goldberg on the relationship between supercompactness and cardinal-preserving extendibility. As an incidental result, we prove a theorem suggesting that supercompactness is the strongest large-cardinal notion preserved by Radin forcing.   
\end{abstract}

\maketitle

\tableofcontents

\section{Introduction}\label{Introduction}
In his pioneering and  influential work \cite{MagSuper}, Magidor discovered the so-called \emph{identity crises phenomenon} for the classes of measurable, strong compact and supercompact cardinals. An uncountable cardinal $\kappa$ is called \emph{measurable} if there is a $\kappa$-complete ultrafilter over $\kappa$.  Similarly, $\kappa$ is called \emph{strongly compact} (\emph{supercompact}) if for each $\lambda\geq \kappa$ every $\kappa$-complete filter over $\mathcal{P}_\kappa(\lambda):=\{x\s \lambda\mid |x|<\kappa\}$ extends to a $\kappa$-complete (normal) ultrafilter. It is clear 
that every supercompact  is strongly compact and that the latter are measurable. Standard arguments also show that every supercompact is a limit of measurables. 
Way more elusive is whether a supercompact cardinal (resp. a strongly compact) has to be a limit of strongly compacts (resp. measurables). 
In  his groundbreaking work \cite{MagSuper}, Magidor proved  the consistency (relative to ZFC) of the next  configurations; namely, (1) the first strongly compact is the first supercompact; (2) the first strongly compact is the first measurable. This is known as  the \emph{identity crises phenomenon}.

Ever since the study of the identity crises phenomena  have captivated the attention of prominent set theorist, such as Apter, Cummings, Gitik, Magidor, Shelah and Woodin, among others.  Extending Magidor's original theorem from \cite{MagSuper}, Kimchi and Magidor \cite{KimMag} proved the consistency of every strongly compact cardinal being supercompact, except for certain limits of strong compacts.
Other major works in the area are due to Apter and Shelah \cite{AptSheII, AptShe}, Apter and Gitik \cite{AptGit}, Apter and Cummings \cite{AptCum, AptCum2} and Woodin \cite{CumWoo}. 
More recently, Hayut, Magidor and the author \cite{HMP} studied the identity crises phenomenon at  higher regions of the large cardinal hierarchy; specifically,  in the region comprised between the first supercompact and 
the so-called \emph{Vop\v{e}nka's Principle}. 

\smallskip

Recall that for  $n<\omega$, $C^{(n)}$ denotes the club class of all ordinals $\theta$ such that $V_\theta$ is a $\Sigma_n$-elementary substructure of $V$ (in symbols,  $V_\theta\prec_{\Sigma_n}\mathrm{V}$). 
A cardinal $\kappa$ is  
\emph{$C^{(n)}$-supercompact} if for all $\lambda>\kappa$ there is an elementary embedding $j\colon V\rightarrow M$ with $\crit(j)=\kappa$, $\lambda<j(\kappa)$, $M^\lambda\s M$ and $j(\kappa)\in C^{(n)}$ \cite{Bag}.   $C^{(0)}$-supercompact cardinals are exactly the usual \emph{supercompact} cardinals. 

Answering a question of Bagaria from \cite[\S5]{Bag},  in \cite{HMP} it was showed that, for each $n<\omega$, the first supercompact cardinal can be the first $C^{(n)}$-supercompact. 
A key component of the proof in \cite[\S3]{HMP} 
is that $C^{(n)}$-supercompat cardinals are preserved by certain products of Prikry-type forcings.  
Unlike with supercompacts or strongly compacts (see \cite[\S6]{Gitik-handbook}), 
$C^{(n)}$-supercompacts are substantially much harder to preserve. 
We refer  to the last section of \cite{HMP} for a thorough discussion on this matter. 

\smallskip

The preservation argument provided in \cite[\S3]{HMP} depends upon  a fairly peculiar feature of Prikry forcing; namely, that generic filters can be constructed  explicitly  via iterated ultrapowers. In fact, no other method is known to preserve $C^{(1)}$-supercompacts.  The downside of this is that the analysis of \cite[\S3]{HMP} is not exportable to asses other plausible configurations for the class of $C^{(n)}$-supercompacts. 
For instance, the strategy of  \cite[\S3]{HMP} fails (dramatically) to make the first two supercompact cardinals be $C^{(1)}$-supercompact.  This is a consequence of a well-known fact  about  Prikry forcing -- namely, that this poset kills all the strongly compacts below the cardinal where the Prikry sequence has been added. As a result neither \cite{HMP} nor the technologies developed so far (e.g. \cite{AptShe, AptCum, AptCum2}) were  amenable to produce the equivalence between supercompactness and $C^{(1)}$-supercompactness. In this paper we pursue a completely different approach. Instead of preserving $C^{(1)}$-supercompact cardinals, first,  we isolate a new axiom --called \emph{$\mathcal{A}$}-- and show that  
it provides the sought one-to-one equiva\-lence. Second, 
 we prove  that $\mathcal{A}$ 
is compatible with  very large cardinals -- specifically, with a proper class of supercompact cardinals. 

\begin{MAIN}\label{MainTheorem}
Suppose the $\rm{GCH}$ holds and that suitable large cardinals exist. Then there is a model of $\mathrm{ZFC}+\mathcal{A}$ with a stationary class of supercompacts. Thus, in this model every supercompact cardinal is $C^{(1)}$-supercompact.
\end{MAIN}

In fact we show that axiom $\mathcal{A}$  is compatible with \emph{Woodin's axiom $I_0$} -- this  entails the existence of much stronger large cardinals than those claimed above  \cite[\S24]{Kan}. 
One  thus infers that  the equivalence between supercompactness and $C^{(1)}$-supercompactness is not precluded  by  the strongest known large cardinal axiom (compatible with the Axiom of Choice). 

\smallskip

The study of axiom $\mathcal{A}$ as well as other  proposed strengthenings (see \S\ref{SectionOpenquestionsaboutW})  has set-theoretic interests on its own. For instance,  $\mathcal{A}$ is intimately connected to the following natural inner-model-theoretic question: 

\begin{question*}
    Suppose $j\colon V\rightarrow M$ is an elementary embedding with $\crit(j)=\kappa$, $M^\kappa\s M$ and $j(\kappa)$ a strong limit  cardinal. Must it be that $\kappa$ is superstrong with target $j(\kappa)$ (possibly witnessed by a different embedding)?
\end{question*}
The above inquires whether the class of \emph{superstrong cardinals} coincides with the \emph{tall cardinals} with strong limit  targets.  This in turn is inspired by the known  equivalence between \emph{strong} and \emph{tall cardinals} in iterable inner models  $L[E]$  proved by Fernandes and Schindler in \cite{FernandesSchidler}. The \emph{Extender Embedding Axiom} (EEA) --due to Woodin (2018)-- postulates an affirmative answer to the above-mentioned question.  
Moreover, Woodin has conjectured that EEA holds in his canonical inner model  for supercompactness \cite{WooBookSuperscompacts}. 
\begin{conjecture*}[Woodin]
$``\mathrm{ZFC}+V=\mathrm{Ultimate-}L$'' proves $\mathrm{EEA}$.
\end{conjecture*}
In contrast, in this paper we  prove two things; namely, (1) axiom $\mathcal{A}$ is consistent with Woodin's axiom $I_0$; (2) the theory $``\mathcal{A}+$There is a supercompact cardinal''  disproves EEA (modulo ZFC). If Woodin's prospects are true, our \textbf{Main Theorem}  provides  
 a fairly exotic model for supercompactness. In \S\ref{sec: ultimate L} we briefly discuss the divergences between $\mathcal{A}$ and $``V=\mathrm{Ultimate-}L$''.




\medskip

In the second part of the manuscript we analyze the relationship between supercompactness and cardinal-preserving extendibility. A cardinal $\kappa$ is said to be \emph{cardinal-preserving extendible}  if for all $\lambda\geq \kappa$ there is an elementary embedding $j\colon V_\lambda\rightarrow M$ with $\crit(j)=\kappa$,  $j(\kappa)>\lambda$ and $\mathrm{Card}^M=\mathrm{Card}\cap M.$  This  notion is due to V. Gitman and J. Osinski. If $j\colon V_\lambda\rightarrow M$ is as above, clearly,  $j(\kappa)$ is a limit cardinal in  $V$. Also, 
every extendible cardinal (see \cite[\S23]{Kan}) is cardinal-preserving extendible. 

\smallskip

In private communication \cite{GG} V. Gitman and G. Goldberg posed us  the following question:
\begin{question*}
Is every supercompact  a  cardinal-preserving extendible?  
\end{question*}
We answer the above  in the negative. Specifically the following is proved: 
\begin{theorem*}
If there is a supercompact cardinal then there is a model  where the first supercompact is not cardinal-preserving extendible. 
\end{theorem*}
The proof of the above  suggests that the strongest large-cardinal notion preserved by Radin forcing is supercompactness (Theorem~\ref{ConsistencyofAdoesnotHold}). Specifically, we show that in the typical Radin extension where a cardinal remains supercompact there are no tall embeddings with a limit cardinal as a target. 

\smallskip

 In \S\ref{Preliminaries} we set notations and provide the reader with some  forcing preliminaries. In \S\ref{WoodinsAxiom} we introduce axiom $\mathcal{A}$ and show that it yields the equivalence between supercompacts and $C^{(1)}$-supercompacts (with inaccessible target). We also show that $\mathcal{A}$ disproves the equivalence betwen superstrong and tall cardinals with strong limit targets postulated by $\mathrm{EEA}$.  The main result of the section is Theorem~\ref{ConsistencyofA} 
which yields the consistency of $\mathcal{A}$ with \emph{Vop\v{e}nka's Principle}. In \S\ref{AandI0} we show that $\mathcal{A}$ is  in fact  compatible with  $I_0$ cardinals. In \S\ref{SectionOpenquestionsaboutW} we  propose a few questions in regards to strengthenings of $\mathcal{A}$.  
Finally, in \S\ref{GitmanGoldbergQ} we answer the  question proposed by Gitman and Goldberg. 

\section{Prelimminaries}\label{Preliminaries}

In this section we garner some definitions and preliminary results that will be used in the proof of our main theorem. We begin with  \S\ref{PreliminariesLargeCardinals} by presenting  a few relevant large cardinal notions.   \S\S\ref{PrelimminariesForcingKurepa} and \ref{SectionCodingPoset} revolve around  two posets instrumental   for the proof of  our main theorem.

\subsection{Some large cardinal notions}\label{PreliminariesLargeCardinals}

Here we collect some relevant large-cardinal notions. For further information on this issue we refer our readers to Kanamori's excellent text on large cardinals \cite{Kan}.
\begin{definition}
     We say that $\kappa$ is  \emph{tall with target $\lambda$} if there is an elementary embedding $j\colon V\rightarrow M$ with $\crit(j)=\kappa$, $M^\kappa\s M$ and $j(\kappa)=\lambda$. A cardinal $\kappa$ is said to be \emph{tall} if $\kappa$ is tall with target $\lambda$ for proper-class many $\lambda>\kappa$.
\end{definition}
 The above notion  admits a natural extender-like characterization; namely,  $\kappa$ is  tall with target $\lambda$ if and only if there is a $(\kappa,\lambda)$-extender $E$ such that $(M_E)^\kappa\s M_E$ and $j_E(\kappa)=\lambda$. For details, see \cite[\S26]{Kan}.
 
 \smallskip
 
 The notion of $C^{(n)}$-tallness is defined analogously by requiring $j(\kappa)$ be a member of the class $C^{(n)}:=\{\theta\in \ord\mid V_\theta\prec_{\Sigma_n} V\}$. Note that $C^{(0)}$ and  $C^{(1)}$ are respectively the class of all ordinals and all  strong limit cardinals. Oftentimes members in $C^{(n)}$ will be referred as \emph{$\Sigma_n$-correct cardinals}. 

 \smallskip

 In this paper we shall be interested in the following version of tallness:
\begin{definition}
	Fix $n<\omega$. A cardinal $\kappa$ is \emph{enhanced $C^{(n)}$-tall}  if for each $\lambda>\kappa$ there is $\theta\in C^{(n)}$ with $\cf(\theta)>\lambda$ such that $\kappa$ is tall with target $\theta$.
\end{definition}

Another important large-cardinal notion is superstrongness:

\begin{definition}
    A cardinal $\kappa$ is called \emph{superstrong with target $\lambda$} if there is an elementary embedding $j\colon V\rightarrow M$ with $\crit(j)=\kappa$, $j(\kappa)=\lambda$ and $V_{\lambda}\s M$. A cardinal $\kappa$ is \emph{superstrong} if it is superstrong with target $\lambda$ for some $\lambda$.
\end{definition}
As before, the superstrongness of a cardinal $\kappa$ can be   characterized via the existence of a $(\kappa,\lambda)$-extender $E$ such that $j_{E}(\kappa)=\theta$ and $V_{\theta}\s M_E$.  

In general we do not require the target model  of a superstrong embedding be closed under $\kappa$-sequences in $V$. Whenever this happens we will say that $\kappa$ is an \emph{enhanced superstrong} cardinal. A key fact about superstrong cardinals with inaccessible targets is the following:
\begin{fact}\label{KeyFactAboutSuperstrong}
    Suppose that $\kappa$ is superstrong with inaccessible target $\lambda$. Then, $\{\tau<\lambda\mid \text{$\kappa$ is superstrong with target $\tau$}\}$ contains a club in $\lambda. $ 
Moreover, $\{\tau<\lambda\mid \text{$\kappa$ is an enhanced superstrong with target $\tau$}\}$ is a ${\geq}\kappa^+$-club.\footnote{A set $C\s \lambda$ is called a \emph{${\geq}\kappa^+$-club} if it is unbounded in $\lambda$ and for every increasing sequence  $\langle \alpha_\xi\mid \xi<\theta\rangle\s C$ with $\theta\in [\kappa^+,\lambda)$ regular, $\sup_{\xi<\theta}\alpha_\xi\in C.$ }
\end{fact}
We refrain to prove the previous fact. For details we invite our readers to consult \cite[Proposition~2.6 and Corollary~2.10]{TsaPhD}.

\smallskip

Other large cardinals considered in the manuscript are \emph{supercompacts}, \emph{cardinal-preserving extendibles},  \emph{$C^{(n)}$-extendibles}, \emph{Vop\v{e}nka cardinals} and \emph{$I_0$ cardinals}. The first two were presented in \S\ref{Introduction} so we are left with the others.

\begin{definition}[Bagaria {\cite{Bag}}]
    A cardinal $\delta$ is called \emph{$\lambda$-$C^{(n)}$-extendible}  for  $\lambda>\delta$ if there is $\theta\in\mathrm{Ord}$ and an elementary embedding $j\colon V_\lambda\rightarrow V_\theta$ such that $\crit(j)=\delta$, $j(\delta)>\lambda$ and $j(\delta)\in C^{(n)}$. A cardinal $\delta$ is called \emph{$C^{(n)}$-extendible} if it is $\lambda$-$C^{(n)}$-extendible for a proper class of $\lambda$.
    \end{definition}
    A cardinal $\delta$ extendible in the classical sense (see \cite{Kan}) if and only if it is $C^{(1)}$-extendible. The next easy observation shows that mild versions of extendibility yield rich superstrong embeddings:
    \begin{prop}\label{Prop1extendibility}
        Suppose that $j\colon V_{\delta+1}\rightarrow V_{j(\delta)+1}$ witnesses $(\delta+1)$-extendi\-bility of $\delta$. Then there is an  embedding $\iota\colon V\rightarrow M$ witnessing that $\delta$ is superstrong with target $j(\delta)$ and $M$ is correct about stationaries of $j(\delta)$.
    \end{prop}
    \begin{proof}
     Let $\mathcal{I}:=V_{j(\delta)}\cup\mathrm{Cub}_{j(\delta)}$ where $\mathrm{Cub}_{j(\delta)}$ denotes the club filter at $j(\delta)$. 
     
     Let us consider the extender $\mathcal{E}:=\langle E_a\mid a\in [\mathcal{I}]^{<\omega}\rangle$ 
where
$$X\in E_a\;\Longleftrightarrow\; X\s [V_{\delta+1}]^{|a|}\,\wedge\, a\in j(X).$$
As usual, there is a natural projection map $\pi_{ab}$ between $E_a$ and $E_b$ whenever $a\s b$ (see \cite[\S23]{Kan}) so we may form the corresponding extender ultrapower $\iota:=j_{\mathcal{E}}\colon V\rightarrow M_{\mathcal{E}}$. It is routine to check that 
$\iota(\delta)=j(\delta)$ and
 $V_{\iota(\delta)}\cup \mathrm{Cub}_{\iota(\delta)}\s M_\mathcal{E}.$ In particular, $M_{\mathcal{E}}$ is stationary correct at $j(\delta)$.
    \end{proof}
    \begin{definition}\label{CorrectSuperstrong}
Superstrong cardinals $\delta$  as those in Proposition~\ref{Prop1extendibility} will be referred as \emph{stationary-correct superstrong (with target $j(\delta)$)}.
\end{definition}
    We will also consider a strengthening of $\Cn$-extendibility: 
\begin{definition}\label{def: aextendible}
  Fix proper classes $A$ and  $\mathbb{P}$. A cardinal $\delta$ is \emph{$C^{(n)}_{\mathbb{P}}$-$A$-extendible} if for all $\lambda>\delta$ there are $\theta\in\mathrm{Ord}$ and an elementary embedding $$j\colon \langle V_\lambda,\in, A\cap \lambda\rangle\rightarrow \langle V_\theta,\in, A\cap \theta\rangle$$ such that $\crit(j)=\delta$, $j(\delta)>\lambda$ and $j(\delta)\in C^{(n)}_{\mathbb{P}}$ where $C^{(n)}_{\mathbb{P}}$ is the club class
$$\{\mu\in\ord\mid \langle V_\mu,\in,\mathbb{P}\cap V_\mu\rangle\prec_{\Sigma_n}\langle V,\in,\mathbb{P}\rangle\}.$$
\end{definition}

\begin{fact}[{\cite[Lemma~3.16]{BP}}]\label{fact: preserving Cns}
    If $\mathbb{P}$ is an adequate class forcing iteration, $j(\delta)\in C^{(n)}_{\mathbb{P}}$ and the trivial condition of $\mathbb{P}$ forces $``V_{j(\delta)}[\dot{G}_{j(\delta)}]=V[\dot{G}]_{j(\delta)}$'' then   $\text{$\one\forces_{\mathbb{P}} ``j(\delta)\in \dot{C}^{(n)}$''}.$
\end{fact}

\smallskip

In \cite{PovOmega} the author introduced the notion of \emph{almost-$C^{(n)}$-extendibility} aiming to asses whether every $C^{(n)}$-extendible is a limit of $C^{(n)}$-supercompact cardinals. This new large-cardinal concept is defined as follows:
\begin{definition}
  A cardinal $\delta$ is called \emph{almost-$C^{(n)}$-extendible} if for all $\lambda>\delta$, $\delta$ is $\lambda$-supercompact and there is $\mu\in C^{(n)}$ with $\cf(\mu)>\lambda$ such that $\delta$ is superstrong with target $\mu.$
\end{definition}
In \cite[\S2]{PovOmega} it is showed that every almost-$C^{(n)}$-extendible cardinal is $\Sigma_{n+2}$-correct. Its important to observe that, for a given almost-$C^{(n)}$-extendible cardinal, its supercompactness and superstrongness are witnessed by (possibly) different elementary embeddings. In the cases where both properties are witnessed by the same embedding, Tsaprounis \cite{Tsan} has showed that this is equivalent to $C^{(n)}$-extendibility. Exploiting this technical nuance the following were proved in \cite[\S2]{PovOmega}:
\begin{enumerate}
    \item Every almost-$C^{(n)}$-extendible  is $C^{(n)}$-supercompact;
    \item Every almost-$C^{(1)}$-extendible is a limit of supercompacts;
    \item Every $C^{(n)}$-extendible is a limit of almost-$C^{(n)}$-extendibles. 
\end{enumerate}
In particular, every $C^{(n)}$-extendible cardinal is a limit of $C^{(n)}$-supercompacts.

\smallskip

Next we discuss Vop\v{e}nka and almost huge cardinals. A set $X\s V_\kappa$ is called \emph{Vop\v{e}nka in $\kappa$}\label{Vopenka} if for every natural sequence of structures\footnote{A natural sequence of structures is a sequence $\langle \mathcal{M}_\alpha\mid \alpha<\kappa\rangle$  where each $\mathcal{M}_\alpha $ takes the form $\langle V_{f(\alpha)},\in,\{\alpha\},R_\alpha\rangle$ with $\alpha<f(\alpha)\leq f(\beta)<\kappa$ and $R_\alpha\s V_{f(\alpha)}$ \cite[p.336]{Kan}.} $\langle\mathcal{M}_\alpha\mid \alpha<\kappa\rangle$  there are ordinals $\alpha<\beta<\kappa$ and an elementary embedding $j\colon \mathcal{M}_\alpha\rightarrow \mathcal{M}_\beta$ with $\crit(j)\in X$.  An inaccessible cardinal $\kappa$ is called \emph{Vop\v{e}nka} if $\kappa$ is Vop\v{e}nka in $\kappa$. The \emph{Vop\v{e}nka filter} on $\kappa$ is $\mathcal{F}:=\{X\s \kappa\mid \text{$\kappa-X$ is not Vop\v{e}nka in $\kappa$}\}.$ It is not hard to show that  $\kappa$ is Vop\v{e}nka if and only if $\kappa$ is inaccessible and $\mathcal{F}$ is a proper filter. Whenever we refer to Vop\v{e}nka's filter we will always be assuming that it is proper. A classical result in set theory says that $\mathcal{F}$ is a normal filter \cite[Proposition~24.14]{Kan} and in particular it contains the club filter on $\kappa$. In special circumstances $\mathcal{F}$ can be extended to a normal (uniform) measure on $\kappa$. The next lemma pinpoints one of those:
\begin{lemma}\label{lemmaalmoshuge}
    If there is an elementary embedding $j\colon V\rightarrow M$ such that $\crit(j)=\kappa$ and $M^{<j(\kappa)}\s M$ then $\mathcal{F}\s\mathcal{U}$ where $\mathcal{U}:=\{X\s\kappa\mid \kappa\in j(X)\}.$
\end{lemma}
\begin{proof}
    By Powell's theorem \cite[Theorem~24.18]{Kan} there is $X\in\mathcal{U}$ such that if $\langle\mathcal{M}_\alpha\mid \alpha<\kappa\rangle$ is a natural sequence of structures and $\alpha<\beta$ are members of $X$ then there is $\iota\colon \mathcal{M}_\alpha\rightarrow\mathcal{M}_\beta$ with $\crit(\iota)=\alpha.$ By elementarity, for each $M$-natural sequence $\langle \mathcal{M}_\alpha\mid \alpha<j(\kappa)\rangle$ and an ordinal $\beta\in j(X)$ with $\kappa<\beta$ there is  an $M$-elementary embedding $\iota\colon \mathcal{M}_\kappa\rightarrow\mathcal{M}_\beta$ with $\crit(\iota)=\kappa$.

    Let $F\in\mathcal{F}$. Then $j(F)$ belongs to the $M$-Vop\v{e}nka filter $j(\mathcal{F})$. That means that there is $\langle \mathcal{N}_\alpha\mid \alpha<j(\kappa)\rangle$ an $M$-natural sequence such that if $\iota\colon \mathcal{N}_\alpha\rightarrow \mathcal{N}_\beta$ is an elementary embedding then $\crit(\iota)\in j(F)$.  Applying the previous observation we conclude that there is an elementary embedding $\iota\colon \mathcal{N}_\kappa\rightarrow\mathcal{N}_\beta$ whose  critical point is $\kappa\in j(F)$. Therefore, $F\in\mathcal{U}$ as needed.
\end{proof}
The assumption employed in the previous lemma is referred in the large-cardinal literature as \emph{almost hugeness.} To grant the above hypothesis as well as some extra properties we will demand a bit more about our $\kappa$:

\begin{definition}\label{def: hugeness}
   A cardinal $\kappa$ is called \emph{huge} if there is an elementary embedding $j\colon V\rightarrow M$ such that $\crit(j)=\kappa$ and $M^{j(\kappa)}\s M$.  
\end{definition}

\smallskip

One of the strongest known large cardinals is \emph{Woodin's axiom $I_0$}:
\begin{definition}[Woodin]\label{DefiningI0}
 $I_0(\lambda)$ denotes the  assertion  ``There is an elementary embedding $j\colon L(V_{\lambda+1})\rightarrow L(V_{\lambda+1})$ with $\crit(j)<\lambda.$''
\end{definition}
To make explicit the critical point of an  $I_0(\lambda)$ embedding $$j\colon L(V_{\lambda+1})\rightarrow L(V_{\lambda+1})$$ we will write $I_0(\delta,\lambda)$ being $\delta:=\crit(j)$. If $I_0(\delta,\lambda)$ holds as witnessed by $j$ we shall denote by $\langle\delta_n\mid n<\omega\rangle$ the \emph{critical sequence of $j$}; namely, the sequence defined by $\delta_0:=\delta$ and $\delta_{n+1}:=j(\delta_n)$ for each $n<\omega$.
\begin{definition}[Woodin]
  For a class $X$ one  says that \emph{$I_0(\delta, \lambda, X)$  holds} if there is an elementary embedding $j\colon L(V_{\lambda+1})\rightarrow L(V_{\lambda+1})$ with $\crit(j)=\delta$ below $\lambda$ and such that $j(X\cap \lambda)=X\cap \lambda$.
\end{definition}


\subsection{A forcing adding a $\kappa$-regressive $\kappa$-Kurepa tree}\label{PrelimminariesForcingKurepa}
Recall that a \emph{tree} is a partially ordered set $\mathbb{T}=(T,<)$ such that for each $t\in T$ the set of predecessors $\mathrm{pred}_{\mathbb{T}}(t):=\{s\in T\mid s<t\}$ is well-ordered by $<$. In a harmless abuse of notation one identifies the tree $\mathbb{T}$ with its underlying universe $T$. For an ordinal $\alpha$, the \emph{$\alpha$th-level level of $T$} is defined to be $$T_\alpha:=\{t\in T\mid \otp(\mathrm{pred}_{\mathbb{T}}(t))=\alpha\}.$$
The collection of  first $\alpha$-many levels $T_{<\alpha}$ is $\{t\in T\mid \otp(\mathrm{pred}_{\mathbb{T}}(t))<\alpha\}$.

The \emph{height of $T$} (denoted $\mathrm{ht}(T)$) is the first ordinal $\alpha$ such that $T_\alpha=\varnothing.$ Finally, a \emph{branch for $T$} is a $\s$-maximal $<$-chain $b\s T$. 

\smallskip

 For the scope of this section let us fix a Mahlo cardinal $\kappa$.\footnote{An inaccessible cardinal that is limit of inaccessibles suffices.} We shall consider $\kappa$-trees on $\kappa$; namely, trees $T\s\kappa$ such that $\mathrm{ht}(T)=\kappa$ and $|T_\alpha|<\kappa$ for all $\alpha<\mathrm{ht}(T)$. The following definition --modulo a slight tweak-- is borrowed from  K\"onig and Yoshinobu's paper \cite{KY}:
 \begin{definition}\label{kurepaforcing}
     A $\kappa$-tree $T$ is said to be a \emph{$\kappa$-regressive $\kappa$-Kurepa tree} if it has $\kappa^+$-many branches and for each inaccessible cardinal $\delta<\kappa$ there is a regressive map $f\colon T_\delta\rightarrow T_{<\delta}$; namely, a map satisfying the following clauses:
     \begin{enumerate}
         \item   $f(x)<_Tx$ for all $x\in T_\delta$;
         \item if $x,y\in T_\delta$ are $<_T$-incompatible then $x\wedge y<_T f(x), f(y)$ where 
         $$x\wedge y:=\sup_{<_T}\{z\in T\mid z<_T x\,\wedge\, z<_T y\}.$$
     \end{enumerate}
 \end{definition}
 \begin{remark}
     The only (yet crucial) difference with the notion in \cite{KY} is that here we require the regressive maps to exist only at inaccessible levels of  $T$. This will be important to ensure that the natural forcing  introducing a $\kappa$-regressive $\kappa$-Kurepa tree is \emph{almost $\kappa$-directed closed}; to wit, for each $\gamma<\kappa$, the poset contains a dense $|\gamma|$-directed-closed subforcing.  As shown in \cite{KY}, the former notion yields a poset which is not almost $\omega_2$-directed-closed for it contradicts $\mathrm{MM}$.  On the other hand, notice that we do not assume anything upon the cardinality of the level sets of $T$. This contrast with the usual demand on $\kappa$-Kurepa trees that $|T_\alpha|\leq|\alpha|+\aleph_0$ for all $\alpha<\kappa.$
 
 \end{remark}
 We wish to define a poset $\mathbb{K}_\kappa$ that kills the measurability of $\kappa$ and it is both $\kappa$-closed and almost $\kappa$-directed-closed. The usual forcing to get a $\kappa$-Kurepa tree (see \cite[\S6]{CumHandBook}) has the first two of these properties but fails to have the third one. It turns out that   this is tightly connected with the requirement that the $\alpha$th-level set of the tree is not too big relative to $|\alpha|$. We do not require this in purpose,  aiming to have some wiggle room for our poset to fulfil the third requirement.  The forcing we shall use is a variation of the one considered by K\"onig and Yoshinobu in \cite[\S3]{KY}. Namely, 
 \begin{definition}\label{Konigposet}
    Let $\mathbb{K}_\kappa$ denote the poset consisting of $p=(t^p, h^p)$ where:
    \begin{enumerate}
        \item $t^p$ is a normal tree on $\kappa$, the levels of $t^p$ have  cardinality ${<}\kappa$ and its height $\mathrm{ht}(t^p)$ is $\alpha^p+1$ for  some $\alpha^p<\kappa$;
        \item for each inaccessible $\delta\leq \alpha^p$  there is a regressive map $f\colon t^p_\delta\rightarrow t^p_{<\delta}$;
        \item $h^p\colon t^p_{\alpha^p}\rightarrow \kappa^+$ is a one-to-one function.
    \end{enumerate}
Given conditions $(t,h), (s,g)$ in $\mathbb{K}_\kappa$   we write $(t, h)\leq (s, g)$ if  $s=t\restriction \mathrm{ht}(s)$, $\ran(g)\s \ran(h)$ and $g^{-1}(\xi)<_t h^{-1}(\xi)$ for each $\xi\in \ran(g)$.
 \end{definition}

 \begin{prop}[Main properties of $\mathbb{K}_\kappa$]\label{PropertiesKonigposet}\hfill
 \begin{enumerate}
     \item $\mathbb{K}_\kappa$ is $\kappa^{+}$-cc and $\kappa$-closed;
      \item $\mathbb{K}_\kappa$ adds a $\kappa$-regressive $\kappa$-Kurepa tree.
     \item $\mathbb{K}_\kappa$ is almost $\kappa$-directed-closed; 
     \item  $\one\forces_{\mathbb{K}_\kappa}``\kappa$ is not measurable''.
 \end{enumerate}
 \end{prop}
 \begin{proof}
    (1) Since $2^{<\kappa}=\kappa$ standard arguments show that $\mathbb{K}_\kappa$ is $\kappa^+$-cc. We skip the verification of $\kappa$-closure as it is similar to the argument for (3).

    (2) Let $G\s\mathbb{K}_\kappa$  a $V$-generic filter and define  $T:=\bigcup\{t^p\mid p\in G\}$.  It is routine to check that $T$ is a $\kappa$-regressive $\kappa$-Kurepa tree as witnessed by the collection of branches $\{b_\alpha\mid \alpha<\kappa^+\}$ defined as $$b_\alpha:=\sup_{<_{T}}\{h^{-1}(\alpha)\mid \exists p\in G\, (h^p=h\,\wedge\, \alpha\in \ran(h^p))\}.$$

    (3) Fix $\gamma<\kappa$ and define $\{(t,h)\in \mathbb{K}_\kappa\mid \text{The height of the tree $t$ is ${>}\gamma$}\}$. Clearly, this set is dense in $\mathbb{K}_\kappa$. We claim that it is also $|\gamma|^+$-directed-closed. 
    
    Let $\{(t_\xi,h_\xi)\mid \xi<|\gamma|\}$ be any directed set of conditions in the set. For each $\xi<|\gamma|$ let $\alpha_\xi+1$ be the height of $t_\xi$ and set $\alpha^*:=\sup_{\xi<\gamma}\alpha_\xi$. Define $$\textstyle t^*:=\bigcup_{\xi<|\gamma|}t_\xi\,\wedge\, <_{t^*}:=\bigcup_{\xi<|\gamma|}<_{t_\xi}.$$
    Let us add to  $(t^*,\leq_{t^*})$ one more level so that the resulting tree has successor height. Let $\mathcal{B}$ denote the collection of all branches $b\s t^*$ (i.e., $\s$-maximal $<_{t^*}$-chains) and  $\Phi\colon \mathcal{B}\rightarrow\kappa\setminus t^*$ be an injective map $b\mapsto \alpha_b$ (note that this exists as $\kappa$ is inaccessible). We declare $\alpha_b$ to be bigger than $\eta$ for all $\eta\in b$ and declare $\alpha_b$ and $\alpha_{b'}$ to be incompatible provided $b\neq b'.$ 

    \smallskip
    
    The above choice yields a new tree $t$ whose first $\alpha^*$-many levels are given by $t^*$ and the top one is given by the new ordinals $\{\alpha_b\mid b\in \mathcal{B}\}$. In symbols, 
    $$t_{{<}\alpha_*}=t^*\;\text{and}\;t_{\alpha_*}=\{\alpha_b\mid b\in\mathcal{B}\}.\footnote{Note that $t_{\alpha_*}$ can possibly have cardinality ${>}|\alpha_*|$ but in any case it has size ${<}\kappa$ and therefore it is a legitimate choice. The fact that $t_{\alpha_*}$ can have any size ${<}\kappa$ is an important difference with  the forcing adding a \emph{slim} $\kappa$-Kurepa tree.}$$

    The proposed one-to-one map   $h\colon t_{\alpha^*}\rightarrow \kappa^+$  is defined as follows. For each ordinal  $\nu\in \bigcup_{\xi<|\gamma|}\ran(h_\xi)$ let us define a $<_{t^*}$-chain through $t^*$ as $$c_\nu:=\sup_{<_{t^*}}\{h_{\xi}^{-1}(\nu)\mid \nu\in \ran(h_\xi)\}.$$
    It is conceivable that $c_\nu$ is not $\s$-maximal but  we can $\s$-extend it to some $b_\nu\in\mathcal{B}$. In that case $\alpha_{b_\nu}$ is $<_t$-bigger than all members of $c_\nu$. Let $$\textstyle \Psi\colon \{\alpha_b\mid b\in\mathcal{B}\}\setminus \{\alpha_{b_\nu}\mid \nu\in \bigcup_{\xi<|\gamma|}\ran(h_\xi)\}\rightarrow \kappa^+\setminus (\sup\bigcup_{\xi<|\gamma|}\ran(h_\xi))+1$$ be any injection. We define $h\colon t_{\alpha^*}\rightarrow\kappa^+$ as follows:
$$h(\alpha_b):=\begin{cases}
    \nu, & \text{if $\alpha_b=\alpha_{b_\nu}$ for some $\nu\in \bigcup_{\xi<|\gamma|}\ran(h_\xi)$;}\\
    \Psi(\alpha_b), & \text{otherwise.}
\end{cases}$$

    \smallskip

    We claim that $(t, h)$ is a condition in $\mathbb{K}_\kappa$. The only caveat is Clause~(2) in Definition~\ref{kurepaforcing}: If $\delta\leq \alpha^*$ is an inaccessible cardinal, since $\cf(\alpha^*)=\gamma<\alpha^*$, it must be the case that $\delta<\alpha^*$. Hence there is some $\xi$ such that $\delta\leq \alpha_\xi$. Since $t_{\delta}=(t_{\xi})_\delta$ and $t_{<\delta}=(t_{\xi})_{<\delta}$ we simply use the fact $(t_\xi, h_\xi)\in\mathbb{K}_\kappa$ to get the sought regressive function. Thus we conclude that  $(t, h)\in \mathbb{K}_\kappa$. The construction of $(t,h)$ guarantees that $(t,h)\leq (t_\xi,h_\xi)$ for all $\xi<|\gamma|.$

    \smallskip

    (4) Suppose otherwise and let $G\s\mathbb{K}_\kappa$ a $V$-generic filter such that $\kappa$ is measurable in $V[G]$. Let $j\colon V[G]\rightarrow M$ be an elementary embedding with $\crit(j)=\kappa$ and $T$ be the $\kappa$-regressive $\kappa$-Kurepa tree introduced by $G$. By elementarity, $M$ models that $j(T)$ is a $j(\kappa)$-regressive $j(\kappa)$-Kurepa tree and that $\kappa$ is inaccessible. Let $ f\colon j(T)_\kappa\rightarrow T$ be a regressive map in $M$. For a branch $b$  of $T$, the $\kappa$th-member of $j(b)$ -- denoted $j(b)_\kappa$-- belongs to $j(T)_\kappa$ and witnesses $f(j(b)_\kappa)<_T j(b)_\kappa$ (by regressiveness). Let $\alpha_b<\kappa$ be such that $$f(j(b)_\kappa)<_Tj(b)_{\alpha_b}=j(b_{\alpha_b})=b_{\alpha_b}.$$
    Let us observe that the map $ b\mapsto b_{\alpha_b}$ is one-to-one: Indeed, if $b\neq b'$ are cofinal branches of $T$ then the regressiveness of $f$ ensures that $f(j(b)_\kappa), f(j(b')_\kappa)$ are $<_T$-above  $j(b)_\kappa\wedge j(b')_\kappa=\sup_{<_T}\{z\in T\mid z\in b\cap b'\}$. So $b_{\alpha_b}\perp_{T} b_{\alpha_{b'}}.$

    This is a contradiction because $T$ has $\kappa^+$-many branches and $|T|=\kappa$.
 \end{proof}

 \subsection{The coding poset}\label{SectionCodingPoset}
Given a regular uncountable cardinal $\chi$ a stationary set $S\s \chi$ is  called \emph{fat} if for every club $C\s \chi$, $S\cap C$ contains closed sets of arbitrarily large order-type below $\chi$. Given a fat stationary set $S\s \chi$, Abraham and Shelah \cite{AbrShe} defined the poset $\mathbb{C}(S)$ which consists of closed bounded sets $c\s S$ and it is ordered by $\sqsupseteq$-extension.  Clearly $\mathbb{C}(S)$ adds a generic club $\mathcal{C}\s \chi$ contained in $S$. The poset is $\chi^+$-cc and $\chi$-distributive -- 
fatness is crucial to establish this latter fact (see \cite{AbrShe} or \cite[\S6]{CumHandBook}).

\smallskip

The following poset will be employed in the forthcoming \S\ref{SectionForcingW}:

\begin{definition}[Coding poset]
    Let $\chi$ be a Mahlo cardinal  and $\varepsilon<\chi$ be a regular cardinal. By $\mathbb{C}(\varepsilon,\chi)$ we denote the two-step iteration $$\textstyle \mathbb{C}(S_{\chi,\varepsilon})\ast(\Col(\varepsilon,{\dot{c}}_0^{++})\times \prod_{\eta<\chi} \Col(\dot{c}_\eta^{++},\dot{c}_{\eta+1}^{++})),$$
    where $S_{\chi,\varepsilon}:=\{\sigma<\chi\mid \sigma\in \mathrm{Sing}\,\wedge\,\sigma>\varepsilon\}$ and
    $\langle \dot{c}_\eta\mid \eta<\chi\rangle$ is a $\mathbb{C}(S_{\chi,\varepsilon})$-name for the corresponding  generic club. 
\end{definition}

The poset $\mathbb{C}(\varepsilon,\chi)$ makes $\chi$ non-Mahlo by introducing a club consisting of singular cardinals and then collapses the gaps in between the members of the generic club. Thus, $\mathbb{C}(\varepsilon,\chi)$ makes $\chi$ be the first inaccessible cardinal above $\varepsilon$. In addition,   $\mathbb{C}(\varepsilon,\chi)$ preserves all cardinals ${\leq}\varepsilon$ and collapses $\varepsilon^+ .$

    Forcing with $\mathbb{C}(\varepsilon,\chi)$ will be instrumental to make the stationary set $\mathcal{S}$ of Lemma~\ref{indiscernibility} definable in our final generic extension. Although $\mathbb{C}(\varepsilon,\chi)$  is not $\chi$-closed it has been posed to be  $\varepsilon$-directed-closed.

   \subsection{Further coding}
    In the proof of Theorem~\ref{ConsistencyofA} and Theorem~\ref{AandI0} we will have to begin with a model carrying an easily-definable well-ordering. This will be used to ensure that the our forcing iteration  is shifted correctly by the relevant extender embeddings $j_{E^*}$ (see Clause~\eqref{C1HI} in page~\pageref{C1HI}).

    \smallskip

    Consider the following minor variation of Hamkins-Reitz {CCA} \cite[\S3]{Reitz}: 
    \begin{definition}
       $\mathrm{CCA}^*$ is the following assertion: For each ordinal $\alpha$ and every set $a\s \alpha$ there is yet another ordinal $\theta$ below the first $\beth$-fixed point above $\alpha$ such that $a=\mathsf{Code}(\alpha,\theta):=\{\beta<\alpha\mid 2^{\aleph_{\theta+\beta+1}}=\aleph_{\theta+\beta+2}\}.$
    \end{definition}
   Under $\mathrm{CCA}^*$ the set-theoretic universe carries a natural well-ordering. Namely,  for each set $a\s \ord$ denote $\alpha_a:=\{\beta\in \ord\mid a\s \beta\}$ and $\theta_a:=\min\{\theta\in \ord\mid a=\mathsf{Code}(a_\alpha,\theta)\}$. Given $a,b\s \ord$ we write $$\text{$a\prec b$ if and only if $\alpha_a<\alpha_b$ or ($\alpha_a=\alpha_b$ and $\theta_a<\theta_b$). }$$
   By coding sets as subsets of their corresponding ranks the above  induces a well-ordering over the set-theoretic universe. This well-ordering is quite absolute; specifically, given an inaccessible cardinal $\lambda$ and sets $a,b\in V_\lambda$,
   \begin{equation*}\;\;\text{$a\prec b$ if and only if $V_\lambda\models``a\prec b$''.}    
   \end{equation*}

   In \cite[\S3]{Reitz}, it is showed that if the GCH holds then $\mathrm{CCA}^*$ holds after forcing with McAloon's class iteration $\mathbb{P}$ \cite{McAloon}. The forcing $\mathbb{P}$ is fairly well-behaved and preserves (virtually) all kinds of large cardinals. For instance forcing with $\mathbb{P}$ preserves hugeness (recall Definition~\ref{def: hugeness}).
   \begin{lemma}\label{lemma: McAloonOK}
       Assume the $\mathrm{GCH}$ holds. If $\kappa$ is a huge cardinal then $$\text{$\one\forces_{\mathbb{P}}``\kappa$ is huge $+\,\mathrm{CCA}^*$''.}$$
   \end{lemma}
   \begin{proof}[Proof sketch]
       Let $j\colon V\rightarrow M$ be a huge embedding. Without loss of generality  members of $M$ take the form $j(f)([\id]_{U})$ where $\id\colon \mathcal{P}(j(\kappa))\rightarrow\mathcal{P}(j(\kappa))$ and $U$ is a normal measure on $\mathcal{P}(j(\kappa))$ (\cite[Theorem~24.8]{Kan}). Let $G\s \mathbb{P}$ a $V$-generic filter. Clearly $j$ lifts to $j\colon V[G_\kappa]\rightarrow M[G_{j(\kappa)}]$. Note that $M[G_{j(\kappa)}]$ is closed under $j(\kappa)$-sequences in $V[G_{j(\kappa)}]$ as $\mathbb{P}_{j(\kappa)}$ is $j(\kappa)$-cc. Next, note that $j``G_{[\kappa,j(\kappa))}$ is a member of $M[G_{j(\kappa)}]$ and a directed subset of $j(\mathbb{P}_{[\kappa,j(\kappa))})$, which is a  $j(\kappa)^+$-closed poset in $V[G_{j(\kappa)}]$.  In particular, there is a master condition $q$ for  $j``G_{[\kappa,j(\kappa))}$. Using the $\mathrm{GCH}_{\geq j(\kappa)}$ and the fact that each maximal antichain $A\in M$ for $j(\mathbb{P}_{[\kappa,j(\kappa))})$ can be represented as $j(f)([\id]_U)$ for $f\colon \mathcal{P}(j(\kappa))\rightarrow \mathcal{P}_{j(\kappa)}(\mathbb{P}_{[\kappa,j(\kappa))})$ we find $q\in H\in V[G_{j(\kappa)}]$ that is $M[G_{j(\kappa)}]$-generic for $j(\mathbb{P}_{[\kappa,j(\kappa))})$. Thus, inside $V[G_{j(\kappa)}]$ we lift our embedding to $j\colon V[G_{j(\kappa)}]\rightarrow M[G_{j(\kappa)}\ast H]$. Finally, we use that $j$ has width\footnote{\label{Footnote: width}The \emph{width} of an elementary embedding $j\colon M\rightarrow N$ between transitive models is said to be ${\leq}\theta$ if every member of $N$ can be expressed as $j(f)(a)$ for  $a\in N$ and a function $f\in M$  with $M\models |\dom(f)|\leq \theta$.} ${\leq}j(\kappa)$ and that $\mathbb{P}_{[j(\kappa),\ord)}$ is $j(\kappa)^+$-distributive to transfer the generic $G_{[j(\kappa),\ord)}$ (see \cite[\S15]{CumHandBook}). This yields a $V[G]$-definable embedding
       $j\colon V[G]\rightarrow M[G_{j(\kappa)}\ast H\ast j``G_{[j(\kappa),\ord)}]$ which witnesses hugeness of $\kappa$.
   \end{proof}
   Forcing with $\mathbb{P}$ preserves even stronger large cardinals; such as the existence of an $I_0(\kappa)$ (or even an $I_0^\sharp(\kappa)$) embedding (see e.g., \cite{DimonteFriedman}).

\section{Axiom $\mathcal{A}$ and its consequences}\label{WoodinsAxiom}
In this section we  introduce axiom $\mathcal{A}$ and show that it yields the sought equivalence between supercompactness and $C^{(1)}$-supercompactness. We also show that $\mathcal{A}$ disproves Woodin's EEA (see Definition~\ref{EEA}). 

\begin{definition}\label{AxiomW}
Axiom $\mathcal{A}$  is the  conjunction of the following statements:
\begin{enumerate}
    \item There is a proper class of inaccessible cardinals. 
    \item Suppose that $\delta$ is  stationary-correct superstrong with inaccessible target $\lambda$.\footnote{Recall Definition~\ref{CorrectSuperstrong}.} Then, $\delta$ is tall with target $\theta$ for all successor inaccessible  cardinals $\theta\in (\delta,\lambda)$.
\end{enumerate}
\end{definition}
\begin{remark}
  Axiom $\mathcal{A}$ has been  distilled as a result of several discussions between the author and Woodin. The author wishes to thank Prof. Woodin for the stimulating conversations on the topic and for his unvaluable insights.
\end{remark}

The next provides a characterization of $\Cn$-supercompact cardinals:

\begin{theorem}\label{CharacterizingCnsupercompacts}
	Fix $n<\omega$. The following are equivalent: 	
	\begin{enumerate}
		\item $\kappa$ is $C^{(n)}$-supercompact; 
		\item $\kappa$ is both supercompact and  enhanced $C^{(n)}$-tall.
	\end{enumerate} 
\end{theorem}
\begin{proof}
	The implication $(1)\Rightarrow(2)$ is evident. As for $(2)\Rightarrow(1)$ one argues as follows. Fix $\lambda>\kappa$ with $\lambda^{<\kappa}=\lambda$. Since $\kappa$ is an enhanced $C^{(n)}$-tall cardinal  there is $\theta\in C^{(n)}$ with $\cf(\theta)>\lambda$ such that $\kappa$ is tall with target $\theta$. Next,  let $j\colon V\rightarrow M$ be the elementary embedding induced by a supercompact measure on $\mathcal{P}_\kappa(\lambda)$. Using that $\cf(\theta)>\lambda=\lambda^{<\kappa}$ standard arguments yield $j(\theta)=\theta$. Thus, by elementarity,  $M$ thinks that $j(\kappa)$ is tall with target $\theta$. Let $E\in M$ be a witnessing $(j(\kappa),\theta)$-$M$-extender and let $i\colon M\rightarrow N\simeq \Ult(M,E)$ be its induced ultrapower. Clearly, the composition $\iota:=i\circ j$ yields an elementary embedding $\iota\colon V\rightarrow N$ with $\iota(\kappa)=\theta\in C^{(n)}$. It remains to show that $N^{\lambda}\cap V\s N$: Let $\langle x_\alpha\mid \alpha<\lambda\rangle\in N^\lambda\cap V$. By definition of $N$, there are $s_\alpha\in [\theta]^{<\omega}$ and functions $f_\alpha\colon [\zeta]^{|s_\alpha|}\rightarrow M$ in $M$ such that $i(f_\alpha)(s_\alpha)=x_\alpha$. Since $M^\lambda\cap V\s M$ it follows that $\langle f_\alpha\mid \alpha<\lambda\rangle\in M$ and thus $i(\langle f_\alpha\mid \alpha<\lambda\rangle)=\langle i(f_\alpha)\mid \alpha<\lambda\rangle\in N$. Similarly, $\langle s_\alpha\mid \alpha<\lambda\rangle\in (H_\theta^M)^{<\lambda}\s H^M_\theta=H^N_\theta$ in that $\theta$ is a strong limit cardinal with cofinality ${>}\lambda$. From these two facts we conclude that $\langle x_\alpha\mid \alpha<\lambda\rangle\in N$, as needed.
\end{proof}

Part of our interest in $\mathcal{A}$ rests on the following key fact:
\begin{theorem}\label{AimpliesCn}
Assume $\mathcal{A}$ holds. Every supercompact cardinal is  enhanced $C^{(1)}$-tall.  In particular, every supercompact cardinal is  $C^{(1)}$-supercompact.

Moreover, under $\mathcal{A}$, if $\delta$ is supercompact and $\delta<\lambda$  then
$$\text{$\delta$ is $\lambda$-supercompact  if and only if $\delta$ is $\lambda$-$C^{(1)}$-supercompact}.$$
\end{theorem}
\begin{proof}
Let $\delta$ be a supercompact cardinal and fix $\lambda>\delta$. Let $\theta$ be  a successor inaccessible 
above $\lambda$ -- this choice is possible by Clause~(1) of axiom $\mathcal{A}$. 

Let $j\colon V\rightarrow M$ be a witness for the  $\theta^+$-supercompactness of $\delta$. Then, 
$$M\models``\text{$\mathcal{A}\,\wedge\,\delta$ is $(\delta+1)$-extendible with target $j(\delta)$''.}$$
The first property follows from elementarity of $j$; 
the second from  $j\restriction V_{\delta+1}$ being a member of $M$ (by $\theta^+$-closure of $M$). In particular Proposition~\ref{Prop1extendibility} implies that $\delta$ is (inside $M$) stationary-correct superstrong with target $j(\delta)$.

From the perspective of $M$, $\theta$ is also a successor inaccessible. Invoking $\mathcal{A}$ inside $M$ the model thinks that   $\delta$ is tall with target $\theta$ and the same property holds true in $V$: Indeed, let 
 $E\in M$ be a $(\delta,\theta)$-$M$-extender witnessing this. Since $M^{\theta^+}\s M$  it follows that $E$ is a $(\delta,\theta)$-$V$-extender  and its $V$-ultrapower embedding $j_E\colon V\rightarrow M_E$ witnesses that $\delta$ is tall with target $\theta$. The argument in Theorem~\ref{CharacterizingCnsupercompacts} shows that $\delta$ is $\lambda$-$C^{(1)}$-supercompact.
\end{proof}
\begin{remark}
    Note that $\mathcal{A}$ actually implies that  every supercompact cardinal is $C^{(1)}$-supercompact with inaccessible targets.
\end{remark}

The next observation shows that under $\mathcal{A}$ every supercompact is $C^{(n)}$-supercompact, for each $n\geq 1$, in certain inner model. 
\begin{prop}
    Assume $\mathcal{A}$ holds. If $\kappa$ is supercompact then there is $M$ an inner model of $\mathrm{ZFC}$ where, for each $n\geq 1$, $M\models \kappa$ is $C^{(n)}$-supercompact.
\end{prop}
\begin{proof}
    Under $\mathcal{A}$ if $\kappa$ is supercompact then it  is $C^{(1)}$-supercompact with inaccessible target $j(\kappa)$ (Theorem~\ref{AimpliesCn}). By Tsaprounis \cite[Proposition~2.8]{TsaChain}, $$C:=\{h(\kappa)<j(\kappa)\mid \text{$\kappa$ is tall with target $h(\kappa)$} \}$$
    contains a ${{\geq}\kappa^+}$-club. 
   Since for each $\lambda<j(\kappa)$ and $n<\omega$ the set $$C\cap C^{(n)}\cap  E^{j(\kappa)}_{>\cf(\lambda)}$$ is non-empty it follows that $$\text{$V_{j(\kappa)}\models ``\ZFC+\kappa$ is supercompact and enhanced $\Cn$-tall''}.$$ By Theorem~\ref{CharacterizingCnsupercompacts} $V_{j(\kappa)}$ models that $\kappa$ is $C^{(n)}$-supercompact.  
\end{proof}

Another interesting aspect of axiom $\mathcal{A}$ finds its  roots in the inner model program at the level of  supercompactness \cite{WooBookSuperscompacts}.  In \cite{FernandesSchidler} Fernandes and Schindler showed that if there is no inner model with a Woodin cardinal and  $L[E]$ is iterable  then, in $L[E]$, a cardinal is tall if and only if it is either strong or a measurable limit of strong cardinals. A related question is whether a similar equivalence is available between the classes of superstrong cardinals  and tall cardinals with strong limit targets. Especially interersting is whether this equivalence is compatible with the existence of a supercompact.  The said equivalence is postulated by the \emph{Extender Embedding Axiom}: 
\begin{definition}[Woodin, 2018]\label{EEA}
    The \emph{Extender Embedding Axiom} (EEA) says that every cardinal $\kappa$ carrying an elementary embedding $j\colon V\rightarrow M$ with $\crit(j)=\kappa$, $M^\kappa\s M$  and $j(\kappa)\in C^{(1)}$ is superstrong with target $j(\kappa).$
\end{definition}
Observe that the superstrong embeddings postulated by EEA might be different from the departing $j\colon V\rightarrow M$. In fact, a version of EEA saying that each $j\colon V\rightarrow M$ as above witnesses superstrongness is inconsistent. For instance, this version of EEA would imply that every $C^{(n)}$-supercompact cardinal is $C^{(n)}$-extendible, which is not possible by virtue of \cite{PovOmega}.

\smallskip

At the present time is unclear to the author whether EEA is consistent with  a supercompact cardinal -- nevertheless, Woodin has conjectured that it will follow from his axiom $V=\mathrm{Ultimate-}L$ (see \cite{midrasha}):

\begin{conjecture*}[Woodin, 2018]
$``\mathrm{ZFC}+V=\mathrm{Ultimate-}L$'' proves $\mathrm{EEA}$.
\end{conjecture*}

Contingent to Woodin's  conjecture, the next shows that under  $V=\mathrm{Ultimate-}L$, $C^{(n)}$-supercompacts are  almost-$C^{(n)}$-extendibles in disguise.  
\begin{theorem}[{$\mathrm{EEA}$}]\label{AprecludesEEA} Every $C^{(n)}$-supercompact is almost-$C^{(n)}$-extendible. 
   
   In particular, $\mathcal{A}+``$There is a supercompact cardinal'' disproves $\mathrm{EEA}$. 
\end{theorem}
\begin{proof}
    Let $n<\omega$ and $\delta$ be a $C^{(n)}$-supercompact cardinal. Fix an arbitrary $\lambda>\delta$. Clearly, $\delta$ is $\lambda$-supercompact. Also, since $\delta$ is $\lambda$-$C^{(n)}$-supercompact, EEA yields an embedding $j\colon V\rightarrow M$ with $\crit(j)=\delta$, $j(\delta)\in C^{(n)}$, $\cf(j(\delta))>\lambda$ and $V_{j(\delta)}\s M$. Combining these two facts we conclude that $\delta$ is almost-$C^{(n)}$-extendible. For the second claim invoke \cite[Proposition~2.3]{PovOmega} to infer that every almost-$C^{(1)}$-extendible is a limit of supercompacts. Under EEA this is equivalent to saying that every $C^{(1)}$-supercompact is a limit of supercompacts, which is not the case under $\mathcal{A}$ (by Theorem~\ref{AimpliesCn}).
\end{proof}
In spite $V=\mathrm{Ultimate-}L$ may entail EEA --which in turn rules out $\mathcal{A}$-- in the next section we prove that $\mathcal{A}$ is consistent with  a proper class of supercompact cardinals -- and in fact with $I_0$ cardinals.

\section{The proof of the main theorem}

In this section we prove the consistency of $\mathcal{A}$   with strong large cardinals.  

\begin{theorem}\label{ConsistencyofA}
    Assume  the $\mathrm{GCH}$ holds and that $\kappa$ is a huge cardinal. Then, the theory $\text{$``\mathrm{ZFC}+\mathcal{A}+$Vop\v{e}nka's Principle''}$ is consistent. 
\end{theorem}
Recall that \emph{Vop\v{e}nka's Principle}  is the schema asserting that every class of structures $\mathscr{C}$ in the same language does carry different $\mathcal{M},\mathcal{N}\in\mathscr{C}$ and an elementary embedding $j\colon \mathcal{M}\rightarrow\mathcal{N}.$ For an inaccessible cardinal $\kappa$, $V_\kappa\models \mathrm{VP}$ is equivalent to $\kappa$ being a Vop\v{e}nka cardinal (see p.~\pageref{Vopenka}). By virtue of Bagaria's \cite[Corollary~4.15]{Bag} the model of Theorem~\ref{ConsistencyofA} is plenty of supercompact cardinals -- in fact, it contains a proper class of $C^{(n)}$-extendible cardinals, for all $n\geq 1$. Thus, the consistency of $\mathcal{A}$ is not precluded by the strongest known large cardinals. In other words, axiom $\mathcal{A}$ is consistent with a mathematical universe rich and intricate (and this is very far from $\mathrm{Ultimate-}L$).
\subsection{Forcing  $\mathcal{A}$}\label{SectionForcingW}
By  forcing preliminary with McAloon iteration (Lemma~\ref{lemma: McAloonOK}) we may assume that our ground model carries a well-ordering $\prec$ satisfying
\begin{equation}\label{eq: dagger}
   \tag{$\dagger$}
  a  \prec b\;\Longleftrightarrow\; V_\lambda\models\text{$``a\prec b$'',}
\end{equation}
for each $a,b\in V_\lambda$ and $\lambda$ inaccessible. The absoluteness of $\prec$ will be crucial in the design of the iteration leading to a model of axiom $\mathcal{A}.$

The proof of Theorem~\ref{ConsistencyofA} will be broken into a series of lemmas. Let us fix  a normal measure $\mathcal{U}$  on $\kappa$ as in the statement of Lemma~\ref{lemmaalmoshuge}. 
\begin{lemma}\label{indiscernibility}
    There is $\mathcal{S}\in \mathcal{U}$ consisting of measurable cardinals whose members are indiscernibles 
    in the following strong sense: Let $\Phi(x,x_0,\dots, x_{n})$ be a first order formula in $\mathcal{L}_\in$, 
    $\alpha_0<\cdots<\alpha_n$ and $\beta_0<\cdots<\beta_n$  members of $\mathcal{S}$ and $a\in V_\kappa$ whose rank is less than $\min(\{\alpha_i\mid i\leq n\}\cup\{\beta_i\mid i\leq n\}).$ Then,
   $$ \text{$\langle V_\kappa,\in\rangle\models \Phi(a, \alpha_0,\dots, \alpha_n)\,\Longleftrightarrow\, \langle V_\kappa,\in\rangle\models \Phi(a, \beta_0,\dots, \beta_n).$}$$ Moreover, $\{\delta<\kappa\mid \text{$\mathcal{S}\cap \delta$ is stationary in $\delta$}\}\in \mathcal{U}$.
\end{lemma}
\begin{proof}
    Let $\langle \Phi_n\mid n<\omega\rangle$ be an enumeration of the first order formulae in  the language of set theory. 
    For each $a\in V_\kappa$   define  a coloring $c_a\colon [\kappa]^{<\omega}\rightarrow 3$ as 
$$c_a(\vec\alpha):=\begin{cases}
    0, & \text{if $\rank(a)\geq \min(\vec\alpha)$};\\
        1, & \text{if $\rank(a)< \min(\vec\alpha)$ and $\langle V_\kappa,\in\rangle\models \Phi(a,\vec\alpha)$};\\
         2, & \text{if $\rank(a)< \min(\vec\alpha)$ and $\langle V_\kappa,\in\rangle\models \neg \Phi(a,\vec\alpha)$};
    \end{cases}
    $$
    By Rowbottom's theorem \cite[Theorem~7.17]{Kan}, for each $a\in V_\kappa$, there is  a $c_a$-homogeneous set $H_a\in \mathcal{U}$; to wit, for each $n<\omega$, $c_a\restriction [H_a]^n$ is constant.  Since $\mathcal{U}$ is uniform it is evident that $c_a`` [H_a]^n\neq \{0\}$ for all $n<\omega$.
    
    Define $H:=\{\alpha<\kappa\mid \forall a\in V_\kappa\,(\rank(a)<\alpha\,\Rightarrow\, \alpha\in H_a)\}.$ It is routine to check that $H\in \mathcal{U}$ and that members of $H$ are indiscernibles in the desired sense. Finally, the moreover part follows from normality of $\mathcal{U}$: If $j\colon V\rightarrow M$ is the ultrapower induced by $\mathcal{U}$ then $\text{$M\models``j(\mathcal{S})\cap\kappa$ is stationary in $\kappa$''}$. 
\end{proof}

For the rest of this section we adopt the following  notations:
\begin{notation}\label{ConvenientNotation}
    For each inaccessible  $\varepsilon<\kappa$   let us denote:
    
   $\br$ $\lambda_\varepsilon$ the next member of $\mathcal{S}$ past $\varepsilon$; 

    $\br$ $\mathscr{E}_{\varepsilon, \mathcal{S}}:=\{E\in V_\kappa\mid \exists\lambda\in\mathcal{S}\, (V_\kappa\models \Phi(E,\varepsilon, \mathcal{S}\cap \varepsilon, \lambda))\}$ where $\Phi(E,\varepsilon, \mathcal{S}\cap \varepsilon, \lambda)$ is the conjunction of the following two sentences:
        \begin{enumerate}
            \item $E$ is an extender with length $\varepsilon$ and strength $\varepsilon$. 
    \item There is a superstrong extender $F$ such that
    \begin{itemize}
        \item  $F\restriction\varepsilon=E$ (in particular, $\crit(E)=\crit(F)$);
        \item $j_{F}(\crit(E))=\lambda;$
        \item $(M_F)^{\crit(F)}\s M_F$;
        \item   $j_F(\mathcal{S})\cap (\varepsilon+1)=\mathcal{S}\cap (\varepsilon+1)$.\footnote{Note that $j_F(\mathcal{S})\cap (\varepsilon+1)$ is the same as $j_F(\mathcal{S}\cap\varepsilon)\cap (\varepsilon+1)$. As a result, the above statement can be easily stated in terms of the parameters $\varepsilon$ and $\mathcal{S}\cap\varepsilon.$}
    \end{itemize}

        \end{enumerate}

        $\br$ For $i\in\{0,1\}$,  $``\sigma\in^1\mathcal{S}$'' means $``\sigma\in\mathcal{S}$'' and $``\sigma\in^0\mathcal{S}$'' means $``\sigma\notin\mathcal{S}$''
\end{notation}
The next lemma reveals the role of indiscernibility of cardinals in $\mathcal{S}$:
\begin{lemma}\label{LemmaAboutPhi}
Let $\varepsilon<\kappa$ inaccessible and $E\in V_\kappa$ an extender with length and strength $\varepsilon$. Then, 
$$E\in \mathscr{E}_{\varepsilon,\mathcal{S}}\;\Longleftrightarrow\;V_\kappa\models \Phi(E,\varepsilon,\mathcal{S}\cap \varepsilon,\lambda_\varepsilon).$$
Moreover, for each $E\in \mathscr{E}_{\varepsilon,\mathcal{S}}$ there is a $\prec$-minimal extender $E^*$ witnessing it
\end{lemma}
\begin{proof}
    Suppose that $E$ is an extender in $\mathscr{E}_{\varepsilon, \mathcal{S}}$. By definition, there is a target $\lambda\in\mathcal{S}$ witnessing $V_\kappa\models \Phi(E,\varepsilon,\mathcal{S}\cap \varepsilon,\lambda)$. Since members of $\mathcal{S}$ are indiscernibles and $\rank\{E,\varepsilon,\mathcal{S}\cap \varepsilon\}<\lambda_\varepsilon\leq \lambda$ it follows that  $V_\kappa\models \Phi(E,\varepsilon,\mathcal{S}\cap \varepsilon,\lambda_\varepsilon).$ Let $E^*$ be the $\prec$-least extender witnessing $E\in \mathscr{E}_{\varepsilon,\mathcal{S}}$. 
\end{proof}
We will define our main iteration $\mathbb{P}_\kappa$ in a prelimminary generic extension by Cohen forcing $\mathrm{Add}(\omega,1)$. This will ensure that  $\mathrm{Add}(\omega,1)\ast \dot{\mathbb{P}}_\kappa$ admits a \emph{gap below $\omega_1$}. As a result, superstrong embeddings in the generic extension by $\mathrm{Add}(\omega,1)\ast \dot{\mathbb{P}}_\kappa$ will be liftings of superstrong extenders in $V$. This can be seen as a consequence of \emph{Hamkin's Gap Forcing} theorem \cite{HamGap}. 

\smallskip

Working in a generic extension by ${\mathrm{Add}(\omega,1)}$ (for simplicity, call it $V$) let $$\mathbb{P}_\kappa:=\varinjlim\langle \mathbb{P}_\varepsilon; \dot{\mathbb{Q}}_{\varepsilon}\mid \varepsilon< \kappa\rangle$$ be the Easton-supported iteration defined by induction as follows:
\begin{IH}\label{Def: induction hypothesis}
   For each inaccessible $\sigma<\varepsilon$ with $\sigma\in^i \mathcal{S}$ and a $V$-generic filter $G_\sigma\s \mathbb{P}_\sigma$ we suppose the following hold:

   \medskip

   \underline{\textbf{Trivial case:}} If $V[G_\sigma]\models``\sigma$ is not inaccessible'' then  $(\dot{\mathbb{Q}}_\sigma)_{G_\sigma}:=\{\one\}.$

   \medskip

      \underline{\textbf{Non-trivial case:}} If $V[G_\sigma]\models``\sigma$ is inaccessible'' then we distinguish among two cases seeking to define $(\dot{\mathbb{Q}}_\sigma)_{G_\sigma}.$ Namely,

\medskip
   {\textbf{Case $(\aleph)$:}}\label{IH} If $\mathcal{S}\cap \sigma$ is non-stationary (in $V$) then:

   \smallskip

   $\br$  If $\mathscr{E}_{\sigma,\mathcal{S}}=\emptyset$ then we set $(\dot{\mathbb{Q}}_\sigma)_{G_\sigma}:=\mathbb{K}_\sigma\ast \dot{\mathbb{C}}(\sigma^{+2+i}_V,\lambda_\sigma).$

   \smallskip

   $\br$ If $\mathscr{E}_{\sigma,\mathcal{S}}$ is non-empty then we demand:
\begin{enumerate}
    \item\label{C1HI} For each extender $F\in \mathscr{E}_{\sigma,\mathcal{S}}$, 
    $j_{F^*}(\mathbb{P}_{\crit(F)})_\sigma=\mathbb{P}_\sigma$.
\item\label{C2HI}  For each extender $F\in \mathscr{E}_{\sigma,\mathcal{S}}$, 
    $j_{F^*}(\mathbb{P}_{\crit(F)})/{G}_\sigma\;\;\text{projects onto ${\mathbb{K}}_\sigma$;}$
    \item\label{C3HI}  The $\sigma$th-stage of the iteration is defined as $$(\dot{\mathbb{Q}}_\sigma)_{G_\sigma}={\mathbb{K}}_\sigma\ast \dot{\mathbb{T}}^{\mathrm{NS}}_\sigma \ast \dot{\mathbb{C}}(\sigma^{+2+i}_V,\lambda_\sigma)$$ where
     $$\textstyle \dot{\mathbb{T}}^{\mathrm{NS}}_\sigma=\prod_{F\in \mathscr{E}_{\sigma,\mathcal{S}}}j_{F^*}(\mathbb{P}_{\crit(F)})/({G}_\sigma\ast \dot{K}_\sigma).\footnote{Here the product forcing stands for the full-support product.}$$
    \item\label{C4HI} $\text{$(\dot{\mathbb{Q}}_\sigma)_{G_\sigma}$ is $\sigma$-closed}$ and $\text{$(\dot{\mathbb{Q}}_\sigma)_{G_\sigma}/{K}_\sigma$ is $\sigma^{+2+i}_V$-closed in $V[G_\sigma\ast K_\sigma]$}.$
\end{enumerate}

\medskip

 {\textbf{Case $(\beth)$:}} If $\mathcal{S}\cap \sigma$ is stationary (in $V$) then: 

 \smallskip

  $\br$  If $\mathscr{E}_{\sigma,\mathcal{S}}=\emptyset$ then we set $(\dot{\mathbb{Q}}_\sigma)_{G_\sigma}:={\mathbb{C}}(\sigma^{+2+i}_V,\lambda_\sigma).$

   \smallskip

   $\br$ If $\mathscr{E}_{\sigma,\mathcal{S}}$ is non-empty then we demand:
   \begin{enumerate}
       \item For each $F\in \mathscr{E}_{\sigma,\mathcal{S}}$, 
    $j_{F^*}(\mathbb{P}_{\crit(F)})_\sigma=\mathbb{P}_\sigma.$
      \item  The $\sigma$th-stage of the iteration is defined as $$(\dot{\mathbb{Q}}_\sigma)_{G_\sigma}={\mathbb{T}}^{\mathrm{NS}^+}_\sigma \ast \dot{\mathbb{C}}(\sigma^{+2+i}_V,\lambda_\sigma)$$ where
     $$\textstyle \dot{\mathbb{T}}^{\mathrm{NS}^+}_\sigma=\prod_{F\in \mathscr{E}_{\sigma,\mathcal{S}}}j_{F^*}(\mathbb{P}_{\crit(F)})/{G}_\sigma.$$
    \item  $\text{$(\dot{\mathbb{Q}}_\sigma)_{G_\sigma}$ is $\sigma^{+2+i}_V$-closed}.$ \qed
   \end{enumerate}
   \end{IH}

Let $\mathbb{P}_\varepsilon$ be the direct limit of $\langle \mathbb{P}_\sigma; \dot{\mathbb{Q}}_{\sigma}\mid \sigma< \varepsilon\rangle$ and $G_\varepsilon\s \mathbb{P}_\varepsilon$ a $V$-generic filter. Let us define $(\dot{\mathbb{Q}}_\varepsilon)_{G_\varepsilon}$ complying with the induction hypothesis. As before, if $V[G_\varepsilon]\models``\varepsilon$ is non-inaccessible'' then we continue and declare $(\dot{\mathbb{Q}}_\varepsilon)_{G_\varepsilon}$ to be trivial.  Otherwise $V[G_\varepsilon]\models``\varepsilon$ is inaccessible'' and we distinguish among two cases according to whether $\mathcal{S}\cap \varepsilon$ is $(V$-) stationary or not. 

\medskip

\underline{\textbf{Case $(\aleph)$:}} Suppose that  $\mathcal{S}\cap \varepsilon$ is non-stationary in $V$.  

\medskip

$\br$ If $\mathscr{E}_{\varepsilon,\mathcal{S}}=\emptyset$  then complying with the \textbf{Induction Hypothesis} we let \begin{equation}\label{CaseAleph1}
\tag{$\aleph$1}
(\dot{\mathbb{Q}}_\varepsilon)_{G_\varepsilon}:={\mathbb{K}}_\varepsilon\ast \dot{\mathbb{C}}(\varepsilon^{+2+i}_V ,\lambda_\varepsilon)\;\text{ provided $\varepsilon\in^i\mathcal{S}$.}
\end{equation}
 $\br$ Otherwise,  $\mathscr{E}_{\varepsilon,\mathcal{S}}\neq \emptyset$ and we proceed as follows. For each extender $E$ in $\mathscr{E}_{\varepsilon,\mathcal{S}}$ let $E^*$ be the $\prec$-least completion of $E$ to an enhanced superstrong extender given by Lemma~\ref{LemmaAboutPhi}. Let $j_{E^*}\colon V\rightarrow M_{E^*}$ be the corresponding ultrapower embedding. Say $\delta$ is the critical point of $j_{E^*}$. This embedding witnesses $$\text{$j_{E^*}(\delta)=\lambda_\varepsilon$,  $V_{\lambda_\varepsilon}\s M_{E^*}$, $(M_{E^*})^\delta\s M_{E^*}$ and $j_{E^*}(\mathcal{S})\cap (\varepsilon+1)=\mathcal{S}\cap (\varepsilon+1)$.}$$ 
\begin{lemma}\label{lemma: C1IH}
   $j_{E^*}(\mathbb{P}_\delta)_\varepsilon=\mathbb{P}_\varepsilon$. Thus, \eqref{C1HI} of the \textbf{Induction Hypothesis}.
\end{lemma}
\begin{proof}
    We argue by induction on $\sigma<\varepsilon$. Suppose that  $\mathbb{P}_\sigma=j_{E^*}(\mathbb{P}_\delta)_\sigma$ and respectively denote by $\mathbb{Q}_\sigma$ and $\mathbb{R}_\sigma$ the $\sigma$th-stages of these iterations.  If $V[G_\varepsilon\restriction\sigma]\models ``\sigma$ is non-inaccessible'' then we are done as $\mathbb{Q}_\sigma=\{\one\}=\mathbb{R}_\sigma.$

    Let us assume that $V[G_\sigma]\models ``\sigma$ is inaccessible''. Suppose that $\mathcal{S}\cap \sigma$ is non-stationary (the case where $\mathcal{S}\cap \sigma$ is stationary is treated analogously). Then, $j_{E^*}(S)\cap \sigma$ is also non-stationary in $M_{E^*}$ because $V_{\lambda_\varepsilon}\s M_{E^*}$ and $$j_{E^*}(\mathcal{S})\cap \sigma=\mathcal{S}\cap \sigma.$$ Thus, both iterations define  $\mathbb{Q}_\sigma$ and $\mathbb{R}_\sigma$ according to \textbf{Case $(\aleph)$}.  
    \begin{claim}
        $\lambda_\sigma=(\lambda_\sigma)^{M_{E^*}}$ and $\lambda_\sigma<\lambda_\varepsilon$
    \end{claim}
    \begin{proof}[Proof of claim]
          Suppose that $\lambda_\sigma>\varepsilon$. Since $V[G_\varepsilon\restriction\sigma]\models ``\sigma$ is inaccessible'' it follows that $\lambda_\sigma$ is the first inaccessible past $\sigma$ in $V[G_\varepsilon\restriction(\sigma+1)]$. However, we were assuming that $V[G_\varepsilon]\models``\varepsilon$ is inaccessible'', which yields a contradiction. 
           
         Therefore, $\lambda_\sigma\leq \varepsilon<\lambda_\varepsilon$. Since $j_{E^*}(\mathcal{S})\cap (\varepsilon+1)=\mathcal{S}\cap (\varepsilon+1)$ we have: on the one hand, $\lambda_\sigma\in j_{E^*}(\mathcal{S})\cap (\varepsilon+1)$, hence $(\lambda_\sigma)^{M_{E^*}}\leq \lambda_\sigma$; on the other hand, $(\lambda_\sigma)^{M_{E^*}}\in \mathcal{S}\cap (\varepsilon+1)$, hence $\lambda_\sigma\leq (\lambda_\sigma)^{M_{E^*}}$
    \end{proof}
    \begin{claim}
        $\mathscr{E}_{\sigma,\mathcal{S}}=(\mathscr{E}_{\sigma,j_{E^*}(\mathcal{S})})^{M_{E^*}}$.
    \end{claim}
    \begin{proof}[Proof of claim]
        Let $F$ be an extender with length and strength $\sigma$. Then,
$$F\in \mathscr{E}_{\sigma,\mathcal{S}}\,\Leftrightarrow\, V_\kappa\models \Phi(\sigma,\mathcal{S}\cap \sigma, F, \lambda_\sigma)\,\Leftrightarrow\, V_{\lambda_\varepsilon}\models \Phi(\sigma,\mathcal{S}\cap \sigma, F, \lambda_\sigma).$$
Since $V_{\lambda_\varepsilon}\s M$, $j_{E^*}(\mathcal{S})\cap (\sigma+1)=\mathcal{S}\cap (\sigma+1)$ and $\lambda_\sigma=(\lambda_\sigma)^{M_{E^*}}$ the above is equivalent to saying $F\in (\mathscr{E}_{\sigma,j_{E^*}(\mathcal{S})})^{M_{E^*}}$, which concludes the proof.
    \end{proof}
Suppose first $\mathscr{E}_{\sigma,\mathcal{S}}=\emptyset$. In that case the above claims yield
$$\mathbb{Q}_\sigma=\mathbb{K}_\sigma\ast \dot{\mathbb{C}}(\sigma^{+2+i}_V,\lambda_\sigma)=\mathbb{R}_\sigma.$$
Suppose that  $\mathscr{E}_{\sigma,\mathcal{S}}$ is non-empty. We have to check that $\mathbb{T}^{\mathrm{NS}}_\sigma=(\mathbb{T}^{\mathrm{NS}}_\sigma)^{M_{E^*}}$ which in turn is equivalent to checking that $F^*$ (the $\prec$-least extender witnessing $F\in\mathscr{E}_{\sigma, \mathcal{S}}$) is the same as $(F^*)^{M_{E^*}}$ (the $\prec$-$M_{E^*}$-least extender witnessing $F\in(\mathscr{E}_{\sigma, j_{E^*}(\mathcal{S})})^{M_{E^*}}$). This follows from the fact that $F^*,(F^*)^{M_{E^*}}\in V_{\lambda_\varepsilon}\s M_{E^*}$ and the absoluteness of the well-ordering $\prec$ (see \eqref{eq: dagger} in page \pageref{eq: dagger}).
\end{proof}

\begin{lemma}\label{lemma: C2IH}
   $j_{E^*}(\mathbb{P}_\delta)/G_\varepsilon$ projects onto $\mathbb{K}_\varepsilon$. Thus \eqref{C2HI} of the \textbf{Induction Hypothesis}.
\end{lemma}
\begin{proof}
    Note that our \textbf{Induction Hypothesis} holds up to the critical point of $E^*$ because $\delta<\varepsilon$. By elementarity of $j_{E^*}$ the same induction hypothesis holds in $M_{E^*}$ at least up to $j_{E^*}(\delta)=\lambda_\varepsilon$, which is bigger than $\varepsilon$.

    Working in $M_{E^*}$ we have that $j_{E^*}(\mathcal{S})\cap \varepsilon$ is non-stationary. This is because $V_{\lambda_\varepsilon}\s M_{E^*}$,  $j_{E^*}(\mathcal{S})\cap \varepsilon=\mathcal{S}\cap \varepsilon$ and the latter is  non-stationary in $V$. Therefore we enter the casusitic given by \textbf{Case~$(\aleph)$}. Invoking Clause~\eqref{C3HI} we get, denoting $(\dot{\mathbb{R}}_\varepsilon)_{G_\varepsilon}$ the $\varepsilon$th-stage of the iteration $j_{E^*}(\mathbb{P}_\delta)$, $$M_{E^*}[G_\varepsilon]\models (\dot{\mathbb{R}}_\varepsilon)_{G_\varepsilon}=\mathbb{K}_\varepsilon\ast \dot{\mathbb{T}}^{\mathrm{NS}}_\varepsilon\ast \dot{\mathbb{C}}(\varepsilon^{+2+i}_V,\lambda_\varepsilon).\footnote{Here both $\dot{\mathbb{T}}^{\mathrm{NS}}_\varepsilon$ and $\lambda_\varepsilon$ are computed inside $M_{E^*}[G_\varepsilon]$. In particular, $(\lambda_\varepsilon)^{M_{E^*}[G_\varepsilon]}<\lambda_\varepsilon$.}$$
    Thus, $j_{E^*}(\mathbb{P}_\delta)/G_\varepsilon$ projects to $\mathbb{K}_\varepsilon$.
\end{proof}
Let us now define the $\varepsilon$th-stage of the main iteration $\mathbb{P}_\kappa$.
\begin{definition}[Tail forcing in \textbf{Case}~$(\aleph)$]
   Let $\dot{\mathbb{T}}^{\mathrm{NS}}_{\varepsilon}$ be a $\mathbb{K}_\varepsilon$-name  for the full-support product of the tail forcings $\prod_{E\in \mathscr{E}_{\varepsilon,\mathcal{S}}}j_{E^*}(\mathbb{P}_{\crit(E^*)})/({G}_\varepsilon\ast \dot{K}_\varepsilon).$ 
\end{definition}
Working in $V[G_\varepsilon]$ define: 
\begin{equation*}\label{CaseAleph2}
\tag{$\aleph$2}
(\dot{\mathbb{Q}}_\varepsilon)_{G_\varepsilon}:=\mathbb{K}_\varepsilon\ast \dot{\mathbb{T}}^{\mathrm{NS}}_\varepsilon\ast \dot{\mathbb{C}}(\varepsilon^{+2+i}_V ,\lambda_\varepsilon) \;\text{ provided $\varepsilon\in^i\mathcal{S}$.}
\end{equation*}
Let us check that this definition of $(\dot{\mathbb{Q}}_\varepsilon)_{G_\varepsilon}$ complies with the remaining clauses of the \textbf{Induction Hypothesis}. Clearly, by the way we posed $(\dot{\mathbb{Q}}_\varepsilon)_{G_\varepsilon}$ Clause~\eqref{C3HI} holds. Let us dispose with Clause~\eqref{C4HI}.
\begin{lemma}\label{lemma: C4IH}
    $(\dot{\mathbb{Q}}_\varepsilon)_{G_\varepsilon}$ is $\varepsilon$-closed and $(\dot{\mathbb{Q}}_\varepsilon)_{G_\varepsilon}/K_\varepsilon$ is $\varepsilon^{+2+i}_V$-closed in the model $V[G_\varepsilon\ast K_\varepsilon].$ Thus Clause~\eqref{C4HI} of the \textbf{Induction Hypothesis} holds.
\end{lemma}
\begin{proof}
Fix $E\in \mathscr{E}_{\varepsilon,\mathcal{S}}$. Since $V_{\lambda_\varepsilon}\s M_{E^*}$, $j_{E^*}(\mathcal{S})\cap \varepsilon$ is non-stationary in $M_{E^*}$ as well.  As a result, we can invoke Clause~\eqref{C4HI} of our \textbf{Induction Hypothesis} inside $M_{E^*}$ for the $\varepsilon$th-stage of the iteration $j_{E^*}(\mathbb{P}_\delta)$. Thus we 
infer that $$M_{E^*}[G_\varepsilon]\models \text{$``(\dot{\mathbb{R}}_\varepsilon)_{G_\varepsilon}$ is $\varepsilon$-closed and projects to $\mathbb{K}_\varepsilon$''}$$
 and since $\varepsilon\in^i\mathcal{S}$, $$M_{E^*}[G_\varepsilon]\models \text{$``(\dot{\mathbb{R}}_\varepsilon)_{G_\varepsilon}/\dot{K}_\varepsilon$ is $\varepsilon^{+2+i}_V$-closed''}.$$ 
 In fact, applying the \textbf{Induction Hypothesis} at all the stages $\sigma\in [\varepsilon,\lambda_\varepsilon)$
 \begin{equation*}\label{eqclosure}
 \tag{$\star$} M_{E^*}[G_\varepsilon\ast K_\varepsilon]\models \text{$``j_{E^*}(\mathbb{P}_\delta)/({G}_\varepsilon\ast {K}_\varepsilon)$ is $\varepsilon^{+2+i}_V$-closed''}.\end{equation*}
We cannot outright infer that $(\star)$ is absolute  between $M_{E^*}$ and $V$ -- after all, $M_{E^*}$ is closed only under $\delta$-sequences in $V$. 
\begin{claim} $V[G_\varepsilon\ast K_\varepsilon]\models \text{$``j_{E^*}(\mathbb{P}_\delta)/({G}_\varepsilon\ast {K}_\varepsilon)$ is $\varepsilon^{+2+i}_V$-closed''}.$
\end{claim}
\begin{proof}[Proof of claim]
    
Let  $\langle p_\alpha\mid \alpha<\beta\rangle\in V[G_\varepsilon\ast K_\varepsilon]$ be a decreasing sequence of conditions in the poset with $\beta<\varepsilon^{+2+i}_V$. 
Let us note a few things:
\begin{enumerate}
    \item[($\alpha$)] $V[G_\varepsilon\ast K_\varepsilon]_{\lambda_\varepsilon}=V_{\lambda_\varepsilon}[G_\varepsilon\ast K_\varepsilon]\s M_{E^*}[G_\varepsilon\ast K_\varepsilon]$;
    \item[($\beta$)] $j_{E^*}(\mathbb{P}_\delta)/(G_\varepsilon\ast K_\varepsilon)\s V[G_\varepsilon\ast K_\varepsilon]_{\lambda_\varepsilon}$;
    \item[($\gamma$)] $V[G_\varepsilon\ast K_\varepsilon]_{\lambda_\varepsilon}$ is closed under $\lambda_\varepsilon$-sequences in $V[G_\varepsilon\ast K_\varepsilon]$.
\end{enumerate}
$(\alpha)$ is true in that $\mathbb{P}_\varepsilon\ast \mathbb{K}_\varepsilon$ is small compared to $\lambda_\varepsilon$ and $V_{\lambda_\varepsilon}\s M_{E^*}$; $(\beta)$ is evident; 
$(\gamma)$ is an immediate consequence of the inaccessibility of $\lambda_\varepsilon$ in $V[G_\varepsilon\ast K_\varepsilon]$.
Combining these items we have that $\langle p_\alpha\mid \alpha<\beta\rangle\in M_{E^*}[G_\varepsilon\ast K_\varepsilon]$ so by equation \eqref{eqclosure} above there is $p_\beta\in M_{E^*}[H]$ a lower bound for it. 
\end{proof}
Since the extender $E\in\mathscr{E}_{\varepsilon,\mathcal{S}}$ was arbitrary we conclude that 
$$\textstyle \mathbb{T}^{\mathrm{NS}}_\varepsilon:=\prod_{E\in\mathscr{E}_{\varepsilon,\mathcal{S}}}j_{E^*}(\mathbb{P}_\delta)/(G_\varepsilon\ast K_\varepsilon)$$
is $\varepsilon^{+2+i}_V$-closed in $V[G_\varepsilon\ast K_\varepsilon]$. This combined with the definition of $(\dot{\mathbb{Q}}_\varepsilon)_{G_\varepsilon}$ given in equation \eqref{CaseAleph2} above implies that  
 $(\dot{\mathbb{Q}}_\varepsilon)_{G_\varepsilon}$ is $\varepsilon$-closed and $(\dot{\mathbb{Q}}_\varepsilon)_{G_\varepsilon}/K_\varepsilon$ is $\varepsilon^{+2+i}_V$-closed. Thus we are done with the lemma.
\end{proof}
The above discussion completes the inductive definition of the iteration $\mathbb{P}_\kappa$ should we fell in \textbf{Case $(\aleph)$}. Let us next discuss the other possible case.

\medskip

\underline{\textbf{Case $(\beth)$:}} Suppose that  $\mathcal{S}\cap \varepsilon$ is stationary in $V$. 

\smallskip

$\br$ If $\mathscr{E}_{\varepsilon,\mathcal{S}}=\emptyset$ then  complying with our \textbf{Induction Hypothesis} we define \begin{equation}\label{CaseBeth1}
\tag{$\beth$1}
(\dot{\mathbb{Q}}_\varepsilon)_{G_\varepsilon}:={\mathbb{C}}(\varepsilon^{+2+i}_V ,\lambda_\varepsilon)\;\text{ provided $\varepsilon\in^i\mathcal{S}$.}
\end{equation}
$\br$ Otherwise, $\mathscr{E}_{\varepsilon,\mathcal{S}}\neq \emptyset$ and we proceed as follows.  As before, for each $E\in\mathscr{E}_{\varepsilon,\mathcal{S}}$ we let $E^*$ be the $\prec$-least  extension of $E$ provided by Lemma~\ref{LemmaAboutPhi}. Since $\mathbb{P}_\varepsilon$ is definable using $\mathcal{S}\cap \varepsilon$ and the well-ordering $\prec$, $j_{E^*}$ shifts $\mathcal{S}\cap \varepsilon$ correctly and $\prec$ is sufficiently absolute we get $j_{E^*}(\mathbb{P}_{\crit(E^*)})_\varepsilon=\mathbb{P}_\varepsilon$. 
\begin{definition}[Tail forcing in \textbf{Case}~$(\beth)$]
    Denote by ${\mathbb{T}}^{\mathrm{NS}^+}_{\varepsilon}$  the full-support product of the tail forcings   $\textstyle \prod_{E\in\mathscr{E}_{\varepsilon,\mathcal{S}}}j_{E^*}(\mathbb{P}_{\crit(E^*)})/{G}_\varepsilon.$
\end{definition}
Complying with our \textbf{Induction Hypothesis} we define 
\begin{equation*}\label{CaseBeth2}
\tag{$\beth$2}
(\dot{\mathbb{Q}}_\varepsilon)_{G_\varepsilon}:= {\mathbb{T}}^{\mathrm{NS}^+}_\varepsilon\ast \dot{\mathbb{C}}(\varepsilon^{+2+i}_V ,\lambda_\varepsilon) \;\text{ provided $\varepsilon\in^i\mathcal{S}$.}
\end{equation*}
Arguing exactly as in the previous lemmas one checks that this definition complies with the requirements of the \textbf{Induction Hypothesis}. 

\smallskip

For the reader's benefit we summarize the definition of $(\dot{\mathbb{Q}}_\varepsilon)_{G_\varepsilon}$:
\begin{definition}[Main iteration]\label{DefinitionMainIteration} Let $i\in\{0,1\}$ be such that $\varepsilon\in^i\mathcal{S}$.\hfill

\smallskip

{\textbf{Case $(\aleph)$:}} \label{Casealeph}Suppose that  $\mathcal{S}\cap \varepsilon$ is non-stationary in $V$.  Then, 
$$({\mathbb{Q}}_\varepsilon)_{G_\varepsilon}:=
\begin{cases}
{\mathbb{K}}_\varepsilon\ast \dot{\mathbb{C}}(\varepsilon^{+2+i}_V ,\lambda_\varepsilon), & \text{if $\mathscr{E}_{\varepsilon,\mathcal{S}}=\emptyset$;}\\
{\mathbb{K}}_\varepsilon\ast \dot{\mathbb{T}}^{\mathrm{NS}}_\varepsilon\ast\dot{\mathbb{C}}(\varepsilon^{+2+i}_V ,\lambda_\varepsilon), & \text{if $\mathscr{E}_{\varepsilon,\mathcal{S}}\neq \emptyset$;}
\end{cases}
$$

\smallskip

{\textbf{Case $(\beth)$:}}\label{CaseBeth} Suppose  that $\mathcal{S}\cap \varepsilon$ is stationary in $V$.  Then, 
$$({\mathbb{Q}}_\varepsilon)_{G_\varepsilon}:=
\begin{cases}
 \mathbb{C}(\varepsilon^{+2+i}_V ,\lambda_\varepsilon), & \text{if $\mathscr{E}_{\varepsilon,\mathcal{S}}=\emptyset$;}\\
 {\mathbb{T}}^{\mathrm{NS}^+}_\varepsilon\ast\dot{\mathbb{C}}(\varepsilon^{+2+i}_V ,\lambda_\varepsilon), & \text{if $\mathscr{E}_{\varepsilon,\mathcal{S}}\neq \emptyset$;}
\end{cases}
$$
\end{definition}
This completes the inductive construction and yields $\mathbb{P}_\kappa$.

\smallskip

A few explanations about the design of the iteration $\mathbb{P}_\kappa$ are in order:
\begin{remark}
  \textbf{Case~($\aleph$)} has been posed so as to ensure that $\mathbb{Q}_\varepsilon$ kills the measurability of $\varepsilon$ whenever $\mathcal{S}\cap \varepsilon$ is non-stationary. Therefore if  $\delta$ happens to be measurable in  a generic extension by $\mathrm{Add}(\omega,1)\ast {\mathbb{P}_\kappa}$ then it must be that $\mathcal{S}\cap \delta$ is stationary in $V$.  Also, note that if $\mathcal{S}\cap\varepsilon$ is stationary in $V$ then $\mathbb{P}_\varepsilon$ will be $\varepsilon$-cc (see e.g. \cite[Proposition~{7.13}]{CumHandBook}) and as a result $\mathbb{P}_\kappa$ itself will preserve the stationarity of $\mathcal{S}\cap \varepsilon$ (because $\mathbb{P}_\kappa/G_\varepsilon$ is $\varepsilon$-closed).
\end{remark}
\begin{remark}
    The definition of the posets $\mathbb{T}^{\mathrm{NS}}_\varepsilon$ and $\mathbb{T}^{\mathrm{NS^+}}_\varepsilon$  have been  posed to anticipate the generics for the tail forcings $j_{E^*}(\mathbb{P}_{\crit(E^*)})/G_\varepsilon$. Here $j_{E^*}$ is an extender embedding that we would eventually like to lift. This is akin to the usual bookkeeping procedure involved in classical proofs; such as the consistency of PFA.  We invite our readers to jump to Theorem~\ref{Aholdsintheend} where we show how to use this to get $\mathcal{A}$ in the generic extension by $\mathrm{Add}(\omega,1)\ast \mathbb{P}_\kappa.$ The next pages will be instead devoted   to prove the main properties of $\mathbb{P}_\kappa$.
\end{remark}

The next lemma shows that we can indeed force with the coding poset $\mathbb{C}(\varepsilon^{+2+i}_V ,\lambda_\varepsilon)$ after both $\mathbb{T}^{\mathrm{NS}}_\varepsilon$ and $\mathbb{T}^{\mathrm{NS}^+}_\varepsilon$. We just give details when the iteration chooses the former forcing -- for the latter the argument is analogous.
\begin{lemma}\label{Closure}$\one\forces_{\mathbb{P}_\varepsilon\ast \dot{\mathbb{K}}_\varepsilon}``\dot{\mathbb{T}}^{\mathrm{NS}}_\varepsilon$ preserves the Mahloness of $\lambda_\varepsilon$''.
\end{lemma}
\begin{proof}

 For the scope of this argument our ground model (denoted $W$) will be $V[G]$ being $G\s \mathbb{P}_\varepsilon\ast \mathbb{K}_\varepsilon$ a  $V$-generic filter.  
Recall that members of $\mathcal{S}$ were measurable in $V$. So let us fix $U_\varepsilon$  a normal measure (in $W$) 
on $\lambda_\varepsilon$ and set $$A:=\{\theta<\lambda_\varepsilon\mid \text{$\theta$ is weakly compact and $\theta>\varepsilon$}\}\in U_\varepsilon.$$ 


  We may assume that for each $E\in \mathscr{E}_{\varepsilon,S}$ the set 
  $\Sigma_{E^*}:=j_{E^*}(\mathcal{S})\cap \lambda_\varepsilon$ is unbounded in $\lambda_\varepsilon$ for otherwise the posets $\mathbb{T}^{\mathrm{NS}}_{\varepsilon, E^*}:=j_{E^*}(\mathbb{P}_{\crit(E)})/G$ would be small relative to $\lambda_\varepsilon$ and therefore they will preserve the Mahloness of $\lambda_\varepsilon$. 
  
  Said this let us denote $$\textstyle\text{ $C:=\bigcap_{E\in \mathscr{E}_{\varepsilon,\mathcal{S}}}\acc(\Sigma_{E^*})$,}$$ where $\acc(\Sigma_{E^*})$ are the accumulation points of $\Sigma_{E^*}$.  Note that $C$ is a club in $\lambda_\varepsilon$ in that  $|\mathscr{E}_{\varepsilon,\mathcal{S}}|<\lambda_\varepsilon$.  Let us show that $\lambda_\varepsilon$ remains Mahlo in $W[{\mathbb{T}}_\varepsilon^{\mathrm{NS}}]$:

   \smallskip

   $\br$ First, since  $\lambda_\varepsilon$ is weakly compact in $W$, the poset ${\mathbb{T}}^{\mathrm{NS}}_\varepsilon$ is $\lambda_\varepsilon$-cc: Each factor is $\lambda_\varepsilon$-cc and the support of the product (i.e., $|\mathscr{E}_{\varepsilon,\mathcal{S}}|$) is smaller than $\lambda_\varepsilon$. Thus, by weak compactness of $\lambda_\varepsilon$, $\mathbb{T}_{\varepsilon}^{\mathrm{NS}}$ is $\lambda_\varepsilon$-cc (see  \cite[Proposition~1.1]{CoxLucke}). Thus, $\lambda_\varepsilon$ remains regular in any extension by $\mathbb{T}^{\mathrm{NS}}_\varepsilon:=(\dot{\mathbb{T}}^{\mathrm{NS}}_\varepsilon)_G$.

\smallskip

   $\br$ Let $H\s \mathbb{T}^{\mathrm{NS}}_\varepsilon$ be generic and let us argue that, in $W[H]$, $A\cap C$ is a stationary set of regular cardinals below $\lambda_\varepsilon$. Clearly, $A\cap C$ is stationary in $W[H]$ because it belongs to the normal measure $U_\varepsilon$ and ${\mathbb{T}}^{\mathrm{NS}}_\varepsilon$ is $\lambda_\varepsilon$-cc.
   
   Let $\theta\in A\cap C$. Since $\theta$ is a weakly compact cardinal and an accumulation point of each $\Sigma_{E^*}$ it follows that both $(\mathbb{T}^{\mathrm{NS}}_{E^*,\varepsilon}\restriction\theta)$  and $\prod_{E}(\mathbb{T}^{\mathrm{NS}}_{E^*,\varepsilon}\restriction\theta)$ are $\theta$-cc.

   \smallskip
   
  Let us denote the lower and upper parts of $\mathbb{T}^{\mathrm{NS}}_{\varepsilon}$ as follows:
    \begin{itemize}
        \item $\mathbb{L}^{E^*}_{\theta}:= (\mathbb{T}^{\mathrm{NS}}_{E^*,\varepsilon}\restriction\theta)$ and $\mathbb{L}_\theta:= \prod_{E\in\mathscr{E}_{\varepsilon,\mathcal{S}}}\mathbb{L}^{E^*}_\theta$;
        \item $\dot{\mathbb{U}}^{E^*}_{\theta}:= (\dot{\mathbb{T}}^{\mathrm{NS}}_{E^*,\varepsilon}\setminus \theta)$ and $\dot{\mathbb{U}}_\theta:= \prod_{E\in\mathscr{E}_{\varepsilon,\mathcal{S}}}\dot{\mathbb{U}}^{E^*}_\theta.$
    \end{itemize}
   Following Foreman \cite{ForHandbook}, let 
   $\mathbb{A}^{E^*}_{\theta}$ denote the term-space forcing $\mathbb{A}(\mathbb{L}^{E^*}_{\theta},\dot{\mathbb{U}}^{E^*}_{\theta})$. Namely, conditions in $\mathbb{A}^{E^*}_{\theta}$ are $\mathbb{L}^{E^*}_{\theta}$-names $\sigma$ ordered as follows $$\sigma\leq_{\mathbb{A}^{E^*}_{\theta}}\tau\;\Longleftrightarrow\; \one\forces_{\mathbb{L}^{E^*}_{\theta}}\sigma\leq_{\dot{\mathbb{U}}^{E^*}_{\theta}}\tau.\footnote{To ensure the universe of $\mathbb{A}^{E^*}_\theta$ is a set one takes representatives from each equivalence class of $\tau\sim \sigma\;:\Longleftrightarrow\;\one\forces_{\mathbb{L}^{E^*}_\theta} \sigma=\tau$. In practice we will disregard this technical subtlelty.}$$

   Note that $\mathbb{A}^{E^*}_{\theta}$ is a $\theta$-closed partial order in $W$ -- this being an immediate consequence of the \emph{Forcing Maximality Principle}.  Thus, the full-support product $\mathbb{A}_\theta:=\prod_{E^*}\mathbb{A}^{E^*}_{\theta}$ is $\theta$-closed in $W$. {Since $\mathbb{L}_{\theta}$ is $\theta$-cc and  $ \mathbb{A}_{\theta}$ is $\theta$-closed, Easton's lemma (see \cite[Fact~5.10]{CumHandBook}) yields} $\one\forces_{\mathbb{L}_{\theta}}$ $``\mathbb{A}_{\theta}$ is $\theta$-distributive''. In particular,  $\mathbb{L}_\theta\times \mathbb{A}_\theta$ preserves the regularity of $\theta$. 

   \smallskip

   To complete the proof note the following:
   \begin{enumerate}
        \item[$(\alpha)$] $\mathbb{T}^{\mathrm{NS}}_\varepsilon$ projects onto $\mathbb{L}_\theta$;
       \item[$(\beta)$] $\one\forces_{\mathbb{L}_\theta}\mathbb{T}^{\mathrm{NS}}_\varepsilon/\dot{G}_{\mathbb{L}_\theta}\simeq \dot{\mathbb{U}}_\theta$;
       \item[$(\gamma)$] $\one\forces_{\mathbb{L}_\theta}\mathbb{A}_\theta$ projects to $\dot{\mathbb{U}}_\theta.$
   \end{enumerate}
   $(\alpha)$ is witnessed by $\langle p_\alpha\mid \alpha<|\mathscr{E}_{\varepsilon,\mathcal{S}}|\rangle\mapsto \langle p_\alpha\restriction\theta \mid \alpha<|\mathscr{E}_{\varepsilon,\mathcal{S}}|\rangle$; $(\beta)$ is evident; 
   and $(\gamma)$ is a consequence of the following general fact when applied with respect to $\langle \mathbb{P}_{E^*}\ast\dot{\mathbb{Q}}_{E^*}\mid E\in\mathscr{E}_{\varepsilon,S}\rangle$ where $\mathbb{P}_{E^*}:=\mathbb{L}^{E^*}_{\theta}$ and $\dot{\mathbb{Q}}_{E^*}:=\dot{\mathbb{U}}^{E^*}_{\theta}$.  
   \begin{claim}
For a family of posets        $\langle\mathbb{P}_\gamma\ast \dot{\mathbb{Q}}_\gamma\mid \gamma<\Gamma\rangle$ the trivial condition of $\prod_{\gamma<\Gamma}\mathbb{P}_\gamma$ forces that $\prod_{\gamma<\Gamma}\mathbb{A}(\mathbb{P}_\gamma,\dot{\mathbb{Q}}_\gamma)$ projects to $\prod_{\gamma<\Gamma}\dot{\mathbb{Q}}_\gamma$.
   \end{claim}
   \begin{proof}[Proof of claim]
       Let $K\s\prod_{\gamma<\Gamma}\mathbb{P}_\gamma$ be $V$-generic and for each $\gamma<\Gamma$ denote by $K_\gamma$ the induced $\mathbb{P}_\gamma$-generic. Working in $V[K]$, let us consider the map $$\pi\colon \langle \sigma_\gamma\mid \gamma<\Gamma\rangle\mapsto \langle (\sigma_\gamma)_{K_\gamma}\mid \gamma<\Gamma\rangle.$$
       Clearly, this is well-defined. It is also order-preserving as, for each $\gamma<\Gamma$, if $\one\forces_{\mathbb{P}_\gamma}\sigma_\gamma\leq_{\dot{\mathbb{Q}}_\gamma} \rho_\gamma$ then $(\sigma_\gamma)_{K_\gamma}\leq (\rho_\gamma)_{K_\gamma}.$ Finally, let us show that $\pi$ is a projection. Suppose that $\langle q_\gamma\mid \gamma<\Gamma\rangle\leq \langle (\sigma_\gamma)_{K_\gamma}\mid \gamma<\Gamma\rangle$. Since this is a true statement in $V[K]$ there is $p\in K$ forcing {$``\langle \dot{q}_\gamma\mid \gamma<\Gamma\rangle\leq \langle\sigma_\gamma\mid \gamma<\Gamma\rangle$''.} For each $\gamma<\Gamma$ let us define an auxiliary $\mathbb{P}_\gamma$-name $\tau_\gamma$ as follows: 
$$\tau_\gamma:=\{\langle \sigma, r\rangle\mid (r\parallel p_\gamma\,\Rightarrow\, \sigma=\dot{q}_\gamma)\,\vee\, (r\perp p_\gamma\,\Rightarrow\, \sigma=\sigma_\gamma)\}.$$
By design, $p_\gamma\forces_{\mathbb{P}_\gamma}\tau_\gamma=\dot{q}_\gamma$ and $\one\forces_{\mathbb{P}_\gamma}\tau_\gamma\leq \sigma_\gamma$. Thus, $\langle\tau_\gamma\mid \gamma<\Gamma\rangle$ is stronger than $\langle\sigma_\gamma\mid \gamma<\Gamma\rangle$ and its image under $\pi$ is exactly $\langle q_\gamma\mid \gamma<\Gamma\rangle$.
   \end{proof} 
   Combining $(\alpha)$--$(\gamma)$ we conclude that $\mathbb{T}^{\mathrm{NS}}_\varepsilon$ preserves the regularity of $\theta$. This completes the proof of the lemma.
\end{proof}

Let us fix $G\s\mathrm{Add}(\omega,1)\ast \dot{\mathbb{P}}_\kappa$ a $V$-generic filter. 
\begin{lemma}[Properties of the generic extension]\label{lemmacoding}\hfill
    \begin{enumerate}
   \item  Each member of $\mathcal{S}\cup\{\kappa\}$  
   is inaccessible in $V[G]$. Moreover, every successor inaccessible in $V[G]$ is a  successor member of  $\mathcal{S}$;
        \item $\mathcal{S}$ is definable with ordinal parameters in $V[G]$. Specifically, $$\mathcal{S}=\{\varepsilon<\kappa\mid V[G]\models \varphi(\varepsilon, \varepsilon^{+3}_V , \varepsilon^{+4}_V)\}$$ where $\varphi(\varepsilon, \varepsilon^{+3}_V , \varepsilon^{+4}_V)$ is the conjunction of
        \begin{itemize}
            \item  $``\varepsilon$ is inaccessible'',
            \item  $`` \varepsilon^{+3}_V $ is a cardinal'',
            \item  and $`` \varepsilon^{+4}_V $ is not a cardinal''.
        \end{itemize}
        \item If $\varepsilon<\kappa$ is a $V[G]$-measurable cardinal then  $\mathcal{S}\cap \varepsilon$ is stationary in 
        $$\text{$V$,  $V[G_\varepsilon]$ and  $V[G]$;}$$
    \end{enumerate}
\end{lemma}
\begin{proof}
(1) Since $\mathcal{S}\s\kappa$ is stationary  $\mathbb{P}_\kappa$ is $\kappa$-cc. Thus, $\kappa$ is regular in $V[G]$. The next arguments will show that $\kappa$ is  strong limit  in $V[G]$.

 Let $\langle \theta_\alpha\mid \alpha< \kappa\rangle$ be an enumeration of $\mathcal{S}$. If $\varepsilon<\kappa$ is the first $V$-inaccessible, $V[G_\varepsilon]\models``\varepsilon$ is inaccessible'' and $\lambda_\varepsilon=\theta_0$. Thus, $\mathbb{P}_\varepsilon\ast\dot{\mathbb{Q}}_\varepsilon$ makes $\theta_0$ be the first inaccessible past $\varepsilon$, which is preserved by the rest of the iteration $\mathbb{P}_\kappa/G_{\varepsilon+1}$. Next, suppose that $\langle \theta_\alpha\mid \alpha<\beta\rangle$ remain inaccessible in $V[G]$. If $\beta$ is successor one argues  exactly as before. So suppose $\beta$ is limit.

 \smallskip

 {$\br$ Suppose that $\theta_\beta=\sup_{\alpha<\beta}\theta_\alpha$}. In this case it suffices to check that $$\text{$V[G_{\theta_\beta}]\models ``\theta_\beta$ is regular''.}$$ Towards a contradiction, suppose there is $\lambda<\theta_\beta$ and a cofinal increasing map $f\colon \lambda\rightarrow\theta_\beta$.  Letting $\alpha<\beta$ with $\lambda<\theta_\alpha$ it follows that $f\in V[G_{\theta_\alpha}]$. This is no possible for it is a generic extension by a forcing of size ${<}\theta_\beta$ and $\theta_\beta$ is regular (in fact measurable) in $V$. 

 \smallskip

 $\br$ Suppose that $\theta_\beta$ is greater than $\sigma:=\sup_{\alpha<\beta}\theta_\alpha$. Since $V[G_\sigma]$ thinks that $\theta_\beta$ is measurable  we can let $\varepsilon\in (\sigma,\theta_\beta)$ be the first $V[G_\sigma]$-inaccessible. Clearly $\varepsilon$ is inaccessible in $V[G_\varepsilon]$ and $\theta_\beta=\lambda_\varepsilon$ so  the tail poset $\mathbb{P}_\kappa/G_\varepsilon$ makes $\theta_\beta$ be the first inaccessible past $\varepsilon$. This completes the discussion.

For the moreover part of Clause~(1) suppose that $\theta$ is a successor inaccessible in $V[G]$ and let $\varepsilon$ be its inaccessible predecessor. If $\theta<\lambda_\varepsilon$ then $(\dot{\mathbb{Q}}_\varepsilon)_{G_\varepsilon}$ would collapse $\theta$. Similarly, $\theta$ cannot be larger than $\lambda_\varepsilon$ because this latter remains inaccessible in $V[G]$. Therefore, $\theta=\lambda_\varepsilon\in \mathcal{S}$ and we are done. 

\smallskip

(2) Suppose that $\varepsilon\in \mathcal{S}$. By the  argument in (1), $V[G_\varepsilon]$ thinks that $\varepsilon$ is inaccessible. Thus  we fall  either in \textbf{Case~$(\aleph)$} or \textbf{Case~$(\beth)$} and the design of the iteration 
$\mathbb{P}_\kappa$ ensures that $ \varepsilon^{+3}_V $ is preserved and $ \varepsilon^{+4}_V $ is collapsed.  Conversely, suppose that $\varphi(\varepsilon, \varepsilon^{+3}_V , \varepsilon^{+4}_V )$ holds in $V[G]$. We claim that $V[G_\varepsilon]$ thinks that $\varepsilon$ is inaccessible: Otherwise, $\mathbb{Q}_\varepsilon$ would be the trivial forcing and $\mathbb{P}_\kappa/G_{\varepsilon+1}$ would be (more than) $ \varepsilon^{+4}_V $-closed, hence $\varphi(\varepsilon, \varepsilon^{+3}_V , \varepsilon^{+4}_V )$ would fail. Thus,  $V[G_\varepsilon]$ thinks that $\varepsilon$ is inaccessible and   we fall either in \textbf{Case~$(\aleph)$} or \textbf{Case~$(\beth)$}. Since $ \varepsilon^{+3}_V $ is preserved it must be the case that $\varepsilon\in \mathcal{S}$, as sought. 

  \smallskip

(3)  
Towards a contradiction, suppose that $\mathcal{S}\cap \varepsilon$  is non-stationary in $V[G_\varepsilon]$. Either $\mathcal{S}\cap \varepsilon$ is stationary in $V$ or it is not. In the first scenario, $\mathbb{P}_\varepsilon$ is a $\varepsilon$-cc iteration and as a result $\mathcal{S}\cap \varepsilon$ remains $V[G_\varepsilon]$-stationary -- a contradiction. In the second scenario,  $\mathbb{Q}_\varepsilon$  prevents $\varepsilon$  from being measurable in $V[G]$ -- this is because $\mathbb{Q}_\varepsilon$ forces with $\mathbb{K}_\varepsilon$ (see Proposition~\ref{Konigposet}). This contradicts our departing assumption and establishes that $\mathcal{S}\cap\varepsilon$ is stationary in $V[G_\varepsilon]$. 

Finally, note that $\mathbb{P}_\kappa/G_\varepsilon$ is an $\varepsilon$-closed iteration (see Clause~$(1)_\varepsilon$ in p.\pageref{IH}) so it preserves the stationarity of $\mathcal{S}\cap \varepsilon$  in $V[G]$.
\end{proof}

We are now in conditions to show that axiom $\mathcal{A}$ holds after forcing.

\begin{theorem}\label{Aholdsintheend}
    $\mathcal{A}$ holds in $V[G]_\kappa$.
\end{theorem}
\begin{proof}
First note that $V[G]_\kappa=W[G]$ where $W=V_\kappa$.
Working in $W[G]$, let $\delta$  be a stationary-correct superstrong cardinal with inaccessible target. Fix a witnessing extender ultrapower $j\colon W[G]\rightarrow M_E$ and let $\theta\in (\delta, j(\delta))$ be a successor inaccessible. We denote by $\varepsilon$ the inaccessible predecessor of $\theta$.\footnote{Note that $\varepsilon$ may be $\delta$.} 

\begin{lemma}\label{preparingtheground}\hfill
\begin{enumerate}
    \item $\theta\in \mathcal{S}$ and $\theta=\lambda_\varepsilon$;
    \item $\mathcal{S}\cap \delta$ is stationary in $W$;
\end{enumerate}
\end{lemma}
\begin{proof}
(1) By Lemma~\ref{lemmacoding}(1) $\theta\in \mathcal{S}$. It cannot be that $\lambda_\varepsilon<\theta$ for then $\theta$ would not be the successor inaccessible of $\varepsilon$. Likewise, it can either be that $\theta<\lambda_\varepsilon$ as $\varepsilon$ is inaccessible in $W[G_\varepsilon]$ and in that case the poset $\mathbb{Q}_\varepsilon$ would collapse $\theta$. 
Clause (2) is a consequence of Lemma~\ref{lemmacoding}(3).
\end{proof}



\begin{lemma}\label{prelimminarylemma}
   $\delta$ is enhanced superstrong with target $\lambda\in \mathcal{S}\setminus (\theta+1)$ in $W[G]$. 
\end{lemma}
\begin{proof}
First, note that the poset $\mathrm{Add}(\omega,1)\ast \dot{\mathbb{P}}_\kappa$ admits a \emph{gap below $\omega_1$}; to wit, $|\mathrm{Add}(\omega,1)|<\aleph_1$ and $\one\forces_{\mathrm{Add}(\omega,1)}``\dot{\mathbb{P}}_\kappa$ is $\aleph_1$-strategically-closed''. Thus, by \emph{Hamkins' Gap Forcing Theorem} \cite{HamGap}, 
 $j\restriction W\colon W\rightarrow M:= M_E\cap W$ is amenable to $W$ (i.e., $j\restriction x\in W$ for all $x\in W$) and it is definable within $W$. 

 \smallskip

Let us observe that $j\restriction W$ is a superstrong embedding (in $W$) with target $j(\delta)$. Clearly, $\crit(j\restriction W)=\delta$. In addition we claim that $W_{j(\delta)}\s M$: First, since $j(\delta)$ is a $W[G]$-inaccessible limit of inaccessibles it is  an inaccessible limit of members of $\mathcal{S}$ (by Lemma~\ref{lemmacoding}(1)). Since Easton-supports are taken at stage $j(\delta)$ this yields $\mathbb{P}_{j(\delta)}\s W_{j(\delta)}$. In particular, $W_{j(\delta)}[G_{j(\delta)}]$ makes sense. By our assumption upon $M_E$,  $W[G]_{j(\delta)}\s M_E$ but since the model $W_{j(\delta)}[G_{j(\delta)}]$ is always contained in $W[G]_{j(\delta)}$ we infer that
 $$W_{j(\delta)}\s W_{j(\delta)}[G_{j(\delta)}]\cap W\s  M_E\cap W=M.$$
\begin{claim}
  $j(\mathcal{S}\cap \delta)=\mathcal{S}\cap j(\delta)$.  
\end{claim}
\begin{proof}[Proof of claim]
    By Lemma~\ref{lemmacoding}(2), 
$$\mathcal{S}\cap \delta=\{\varepsilon<\delta\mid W[G]\models \text{$\varphi(\varepsilon,\varepsilon^{+3}_W,\varepsilon^{+4}_W)$}\}.$$
By  \emph{Laver's Ground Model Definability theorem} (see \cite{LavGroundModel}),  $W$ is a class definable  within $W[G]$. Thus, since $j$ moves $W$ to $M$, 
$$j(\mathcal{S}\cap \delta)=\{\varepsilon<j(\delta)\mid M_{E}\models\text{$\varphi(\varepsilon,\varepsilon^{+3}_M,\varepsilon^{+4}_M)$}\}.$$
Since $W[G]_{j(\delta)}\s M_E$ it follows that $\varepsilon^{+3}_M$ (resp. $\varepsilon^{+4}_M$) is a cardinal in $M_E$ if and only if it is so in $W[G]$. Similarly, since $W_{j(\delta)}\s M$ necessarily $\varepsilon^{+3+i}_M=\varepsilon^{+3+i}_W$ for $i\in\{0,1\}$. These two facts combined yield 
\begin{equation*}
    j(\mathcal{S}\cap \delta)=\{\varepsilon<j(\delta)\mid W[G]\models\text{$\varphi(\varepsilon,\varepsilon^{+3}_W,\varepsilon^{+4}_W)$}\}=\mathcal{S}\cap j(\delta).\qedhere
\end{equation*}
\end{proof}
Note that  $\mathcal{S}\cap \delta$ is $W[G]$-stationary as
$\delta$ is measurable in $W[G]$ (Lemma~\ref{lemmacoding}). 
Hence $j(\mathcal{S}\cap \delta)$ is stationary in $M_E$ and since this latter model is stationary-correct at $j(\delta)$,  $\mathcal{S}\cap j(\delta)$ is stationary in $W[G]$, as well. 
Finally, $\delta$ is superstrong with inaccessible target $j(\delta)$ so there is a ${\geq}\delta^+$-club  $C\s j(\delta)$ of possible superstrong targets witnessing that  $\delta$ is enhanced superstrong (see Fact~\ref{KeyFactAboutSuperstrong}). In particular, there must be a cardinal $\lambda\in (C\cap \mathcal{S})\setminus \theta+1$ for which $\delta$ is enhanced superstrong with target $\lambda.$ 
\end{proof}
Let $\iota\colon W[G]\rightarrow M$ be a superstrong embedding with $\crit(\iota)=\delta$, $\iota(\delta)=\lambda\in \mathcal{S}\setminus (\theta+1)$ and $M^\delta\s M$. We may assume, without losing any generality, that $\iota$ is the extender ultrapower by some $F^*\in W[G]$.  Once again by \emph{Hamkins'  Gap Forcing theorem} \cite{HamGap}, $F:=F^*\cap W$ is a supestrong extender in $W$ with $j_{F}(\delta)=\lambda$ and $(M_{F})^\delta\cap W\s (M_F).$ 

The definability of $\mathcal{S}$ again implies that $$j_{F}(\mathcal{S}\cap \delta)=j_{F^*}(\mathcal{S}\cap \delta)=\mathcal{S}\cap \lambda.$$
Letting $E:=F\restriction\varepsilon$ 
it is clear that $E\in \mathscr{E}_{\varepsilon, \mathcal{S}}$ (recall Notation~\ref{ConvenientNotation}). 
By indescernibility (i.e., by Lemma~\ref{LemmaAboutPhi}) there is an extender $E^*$ witnessing
$$\langle V_\kappa, \in \rangle\models \Phi(E, \varepsilon, \mathcal{S}\cap\varepsilon,\lambda_\varepsilon).$$ 
\begin{lemma}\label{endgame}
   $W[G]\models ``\delta$ is tall with target $\theta$''. 
\end{lemma}
\begin{proof}
By Lemma~\ref{preparingtheground},  $\lambda_\varepsilon=\theta.$ So $j_{E^*}\colon W\rightarrow M_{E^*}$ is such  that $j_{E^*}(\delta)=\theta$,  $W_{\theta}\s M_{E^*}$ and $(M_{E^*})^\delta\cap W\s M_{E^*}$. 
By our \textbf{Induction Hypothesis}, $j_{E^*}(\mathbb{P}_\delta)_\varepsilon=\mathbb{P}_\varepsilon$ (see also Lemma~\ref{lemma: C1IH}). In particular, $G_\varepsilon$ is generic for $j_{E^*}(\mathbb{P}_\delta)_\varepsilon$ over $M_{E^*}.$ Clearly, for each $H\s j_{E^*}(\mathbb{P}_\delta)/G_\varepsilon$  generic over $W[G_\varepsilon]$ 
    the embedding $j_{E^*}$ lifts (in $W[G_\varepsilon\ast H]$) to $$j_{E^*}\colon W[G_\delta]\rightarrow M_{E^*}[G_\varepsilon\ast H].$$ We argue that there is $H$ which is internal to our  generic extension $W[G]$. 

    \smallskip
    
    In $W[G_\varepsilon]$, $\varepsilon$ is inaccessible so we fall in one of the cases of Definition~\ref{DefinitionMainIteration}. 
    Suppose for instance that we fall  in \textbf{Case~($\aleph$)} -- the other case is treated analogously.  
    The poset $\mathbb{Q}_\varepsilon$  forces over $W[G_\varepsilon\ast K_\varepsilon]$ with $\mathbb{T}^{\mathrm{NS}}_\varepsilon$. By design, this latter poset projects to $\mathbb{T}^{\mathrm{NS}}_{\varepsilon, E^*}:=j_{E^*}(\mathbb{P}_\delta)/(G_\varepsilon\ast K_\varepsilon)$. Thus, any $W[G_\varepsilon\ast K_\varepsilon]$-generic for $(\mathbb{Q}_\varepsilon)_{G_\varepsilon\ast K_\varepsilon}$ yields a $W[G_\varepsilon\ast K_\varepsilon]$-generic for $j_{E^*}(\mathbb{P}_\delta)/(G_\varepsilon\ast K_\varepsilon)$. Let $H_\varepsilon$ be the induced generic. Clearly, $j_{E^*}$ lifts to $$j_{E^*}\colon W[G_\delta]\rightarrow M_{E^*}[G_\varepsilon\ast K_\varepsilon\ast  H_\varepsilon]$$ and the latter is definable in the generic extension  $W[G_{\varepsilon+1}]$.

   Let us now complete the lifting argument. First,  $(\dot{\mathbb{Q}}_\delta)_{G_\delta}$ is $\delta^+$-closed as $\mathcal{S}\cap \delta$ is stationary in $W$ (see \textbf{Case ($\beth$)} of Definition~\ref{DefinitionMainIteration}).
   Since the width\footnote{For the definition of the \emph{width} of an embedding see Footnote~\ref{Footnote: width}.} of the embedding $j_{E^*}$ is ${\leq}\delta$ and  $\mathbb{P}_\kappa/G_\delta$ is $\delta^+$-closed,  $j_{E^*}$ lifts to $$j_{E^*}\colon W[G]\rightarrow M_{E^*}[G_{\varepsilon}\ast K_\varepsilon\ast  H_\varepsilon][j_{E^*}``G_{[\delta, \kappa)}].$$ For details, see \cite[\S15]{CumHandBook}. This lifting  is completely internal to $W[G]$. 
   
 Since $M_{E^*}$ was closed under $\delta$-sequences in $W$  standard arguments show that $M_{E^*}[G_{\varepsilon}\ast K_\varepsilon\ast H_\varepsilon][j_{E^*}``G_{[\delta, \kappa)}]$ is closed under $\delta$-sequences in $W[G]$, as well.  Altogether, $\delta$ is a tall cardinal with target $\theta$ in $W[G]$. 
\end{proof}
This completes the proof of  axiom $\mathcal{A}$ in $W[G]=V[G]_\kappa$.
\end{proof}

\subsection{$\mathcal{A}$ is consistent with Vop\v{e}nka's Principle}

The main result of this section is Theorem~\ref{VopenkaIntheEnd} where we show that $\kappa$ remains a Vop\v{e}nka cardinal after forcing with $\mathbb{P}_\kappa$. Recall that we were assuming that $\kappa$ is almost huge and that $\prec$ is a absolute-enough well-ordering of the set-theoretic universe.

One could be tempted to invoke Brooke-Taylor's preservation theorem of Vop\v{e}nka cardinals \cite[Theorem~15]{Broo} but there is a technical caveat -- our iteration is not  \emph{progressively-directed-closed}.  This is because at a number of stages we are forcing with the Kurepa-type forcing $\mathbb{K}_\varepsilon$ which  is not $\varepsilon$-directed-closed.  Fortunately, for a $\mathcal{U}$-large set of $\delta<\kappa$ (recall that $\mathcal{F}\s \mathcal{U}$ by Lemma~\ref{lemmaalmoshuge}) $\mathbb{P}_\kappa$ contains a dense sub-iteration $\mathbb{D}^\delta$ such that:
\begin{enumerate}\label{ClausesDirectedClosure}
    \item[($\alpha$)] $\mathbb{D}^\delta$ is definable via $\mathcal{S}$, $\delta$ and $\prec$; 
    \item[($\beta$)] $(\mathbb{D}^\delta)_\delta$ forces the tail forcing $(\mathbb{D}^\delta)_{[\delta, \kappa)}$ to be $\delta^+$-directed-closed.
\end{enumerate}

\begin{lemma}\label{Nicesplitting}
    $\{\delta<\kappa\mid \text{There is $\mathbb{D}^\delta\s \mathbb{P}_\kappa$ with $(\alpha)$ and $(\beta)$ above}\}\in\mathcal{U}.$
\end{lemma}
\begin{proof}
    We show that $\{\delta<\kappa\mid \text{$\mathcal{S}\cap \delta$ is stationary}\}\in\mathcal{U}$ is included in the above-displayed set. First we prove a prelimminary claim:
    \begin{claim}\label{ClaimDefinableDirected}
        For each $\varepsilon<\kappa$ inaccessible and $\sigma<\varepsilon$ there is a $\mathbb{P}_\varepsilon$-name for a poset $\dot{\mathbb{D}}_{\varepsilon,\sigma}$ such that the trivial condition of  $\mathbb{P}_\varepsilon$ forces $$\text{$``\dot{\mathbb{D}}_{\varepsilon,\sigma}\in \dot{\hod}_{\{\sigma,\varepsilon, \mathcal{S},\prec\}}\;\wedge\; \dot{\mathbb{D}}_{\varepsilon,\sigma}\s \dot{\mathbb{Q}}_\varepsilon\;\text{is dense and $|\sigma|^+$-directed-closed}$''}.$$
    \end{claim}
    \begin{proof}[Proof of claim]
        We prove this by induction on $\varepsilon<\kappa$. If  $\varepsilon$ happens to be the first inaccessible cardinal we fall in \textbf{Case~$(\aleph)$} of the definition of $\mathbb{P}_\kappa$ (i.e., $\mathcal{S}\cap \varepsilon$ is not stationary). In those circumstances $\mathbb{P}_\varepsilon$ is trivial and $\mathscr{E}_{\varepsilon,\mathcal{S}}=\emptyset$. Thus, for each $\sigma<\varepsilon$  it suffices to take $$\mathbb{D}_{\varepsilon,\sigma}:=\{(p,\dot{c})\in\mathbb{K}_\varepsilon\ast \mathbb{C}(\varepsilon^{+2+i}_V,\lambda_\varepsilon)\mid \mathrm{ht}(t^p)>\sigma\}.$$ 

        Assume the claim holds for all inaccessibles $\epsilon<\varepsilon.$ Suppose that $\mathcal{S}\cap(\varepsilon+1)$ is non-stationary (the other case is treated analogously). Then, we fall in \textbf{Case~$(\aleph)$} of the definition of $\mathbb{P}_\kappa$ (see Definition~\ref{DefinitionMainIteration}). 
There are two cases to consider; namely, either $\mathscr{E}_{\varepsilon,S}=\emptyset$ or $\mathscr{E}_{\varepsilon,S}\neq \emptyset$. In the former scenario it suffices to take $\mathbb{D}_{\varepsilon,\sigma}$ as in the above-displayed equation. Otherwise, $\dot{\mathbb{Q}}_\varepsilon$ takes the form $\dot{\mathbb{K}}_\varepsilon\ast\dot{\mathbb{T}^{\mathrm{NS}}_\varepsilon}\ast \dot{\mathbb{C}}(\varepsilon^{+2+i}_V,\lambda_\varepsilon).$ Let $E\in \mathscr{E}_{\varepsilon,S}$. Set $\delta:=\crit(E)$. Since $\delta<\varepsilon$ our induction hypothesis applied within $M_{E^*}$ yields, for each $\varepsilon\leq \rho<\lambda_\varepsilon$
$$\one\forces^{M_{E^*}}_{j_{E^*}(\mathbb{P}_\delta)_\rho}``\text{There is $\dot{\mathbb{D}}_{\rho,\sigma}\s \mathbb{Q}_{\rho}$ witnessing the claim''}.$$
Let $\dot{\mathbb{D}}^{E^*}_{\varepsilon,\sigma}$ a $j_{E^*}(\mathbb{P}_\delta)_\varepsilon=\mathbb{P}_\varepsilon$-name for the iteration resulting from replacing each $\mathbb{Q}_\rho$ by $\mathbb{D}_{\rho,\sigma}.$ It turns out that this iteration is now definable using 
$$\varepsilon,\sigma,\prec\;\text{and}\;j_{E^*}(\mathcal{S})\cap \lambda_\varepsilon.$$
In turn, 
\begin{itemize}
    \item $E^*$ is definable using $\mathcal{S}\cap \varepsilon$,  $\varepsilon$ and $\lambda_\varepsilon$;
    \item and $\lambda_\varepsilon$ is definable using $\mathcal{S}$ and $\varepsilon$.
\end{itemize}
Therefore $\dot{\mathbb{D}}^{E^*}_{\varepsilon,\sigma}$ is $\mathbb{P}_\varepsilon$-forced to be definable via $\varepsilon,\sigma,\prec$ and $\mathcal{S}$. By the way the iteration $\dot{\mathbb{D}}^{E^*}_{\varepsilon,\sigma}$ is defined at stage $\varepsilon$ it factors as $$\{p\in\dot{\mathbb{K}}_\varepsilon\mid \mathrm{ht}(t^p)>\sigma\}\ast (\dot{\mathbb{D}}^{E^*}_{\varepsilon,\sigma})^{\mathrm{tail}}.$$

\smallskip

The full-support product of the family $\langle (\dot{\mathbb{D}}^{E^*}_{\varepsilon,\sigma})^{\mathrm{tail}}\mid {E\in\mathscr{E}_{\varepsilon,\mathcal{S}}}\rangle $ is forced to be a dense subforcing of $\dot{\mathbb{T}}^{\mathrm{NS}}_\varepsilon$ which is definable via $\varepsilon,\sigma$ and $\mathcal{S}$. Finally, let $\dot{\mathbb{D}}_{\varepsilon,\sigma}$ be a $\mathbb{P}_\varepsilon$-name for the poset
$$\textstyle  \{p\in\dot{\mathbb{K}}_\varepsilon\mid \mathrm{ht}(t^p)>\sigma\}\ast (\prod_{E\in\mathscr{E}_{\varepsilon,\mathcal{S}}}(\dot{\mathbb{D}}^{E^*}_{\varepsilon,\sigma})^{\mathrm{tail}})\ast \mathbb{C}(\varepsilon^{+2+i}_V,\lambda_\varepsilon).$$
This poset is forced to be dense in $\dot{\mathbb{Q}}_\varepsilon$, $|\sigma|^+$-directed-closed and definable using $\{\sigma,\varepsilon,\mathcal{S}\}$. This completes the verification of the induction step.
    \end{proof}
   For each $\delta$ in the set displayed in the lemma, let  $\mathbb{D}^\delta$ be the iteration which up to $\delta+1$ is defined as $\mathbb{P}$ and for all other inaccessible stages $\delta<\varepsilon$ it forces with $\dot{\mathbb{D}}_{\varepsilon,\delta}.$ It follows that $\mathbb{D}^\delta$ is dense in $\mathbb{P}_\kappa$, it is definable via $\mathcal{S},\prec$ and $\delta$ and $(\mathbb{D}^\delta)_\delta$ forces the tail poset $\dot{\mathbb{D}}^\delta_{[\delta,\kappa)}$ be $\delta^+$-directed closed. For this latter claim we use  that $\mathcal{S}\cap \delta$ is stationary as then we fall in \textbf{Case~($\beth$)} of Definition~\ref{DefinitionMainIteration}, which does not involve the Kurepa-type forcing.
\end{proof}

A second important fact is next lemma. We refer our readers to \S\ref{Preliminaries}, Definition~\ref{def: aextendible}, where we supplied the relevant definitions. 
\begin{lemma}\label{Manyextendibles}
    $\{\delta<\kappa\mid \langle V_\kappa,\in, \mathbb{P}_\kappa, \mathcal{S}\rangle\models\text{$``\delta$ is $C^{(n)}_{\mathbb{P}_\kappa}$-$\mathcal{S}$-extendible''}\}\in\mathcal{U}$.
\end{lemma}
\begin{proof}
Let $\mathfrak{M}:=\langle V_\kappa,\in,\mathbb{P}_\kappa, \mathcal{S}\rangle$ and $f\colon \kappa\rightarrow\kappa$ be the function defined as
    $$f(\delta):=\begin{cases}
        \delta, & \text{if $\mathfrak{M}\models \text{$``\delta$ is $C^{(n)}_{\mathbb{P}_\kappa}$-$\mathcal{S}$-extendible''}$};\\
        \delta+\sigma, & \text{$\sigma$  is least with $\mathfrak{M}\models\text{$``\delta$ is not $(\delta+\sigma)$-$C^{(n)}_{\mathbb{P}_\kappa}$-$\mathcal{S}$-extendible''}$}.
    \end{cases}$$
    The set $$C:=\{\delta<\kappa\mid f``\delta\s \delta\,\wedge\, \langle V_\delta,\in,\mathbb{P}_\kappa \cap V_\delta\rangle\prec_{\Sigma_n}\langle V_\kappa,\in,\mathbb{P}_\kappa\rangle\}$$ is a club because it is the intersection of the closure points of $f$ with the club $(C^{(n)}_{\mathbb{P}_\kappa})^{\langle V_\kappa,\in,\mathbb{P}_\kappa\rangle}$ (by inaccessibility of $\kappa$, \cite[Proposition~6.2]{Kan}). 
    
    By normality of the Vop\v{e}nka filter (\cite[Proposition~24.14]{Kan}) $C\in \mathcal{F}$ so that it belongs to $\mathcal{U}$ as well.  Therefore, there is a natural sequence $\langle \mathcal{M}_\alpha\mid\alpha<\kappa\rangle$ such that every elementary embedding $j\colon \mathcal{M}_\alpha\rightarrow \mathcal{M}_\beta$ has $\crit(j)\in C$. 

     \smallskip

    Denote by $X$ the set displayed in the statement of the lemma. Our goal is to show that $X\in \mathcal{F}$; equivalently, that $\kappa\setminus X$ is not Vop\v{e}nka in $\kappa$. To show this we consider yet another natural sequence $\langle \mathcal{N}_\alpha\mid\alpha<\kappa\rangle$ given by
    $$\mathcal{N}_\alpha:=\langle V_{\gamma_\alpha},\in,\{\alpha\}, \mathcal{M}_\alpha, C\cap\gamma_\alpha,\mathcal{S}\cap\gamma_\alpha\rangle$$
    where $\gamma_\alpha:=\min \{\rho\in C\mid \rho>\sup(\mathcal{M}_\alpha\cap\Ord)\}$. 
    
    Since $\kappa$ is assumed to be almost huge (hence Vop\v{e}nka)  there are $\alpha<\beta<\kappa$ and non-trivial elementary embeddings  $j\colon \mathcal{N}_\alpha\rightarrow \mathcal{N}_\beta$ with $\crit(j)<\kappa$. We will show that the critical point of any $j\colon \mathcal{N}_\alpha\rightarrow \mathcal{N}_\beta$ as above must belong to $X$ -- this will establish that $\kappa\setminus X$ is not Vop\v{e}nka in $\kappa$.

    \smallskip

    Let $j\colon \mathcal{N}_\alpha\rightarrow \mathcal{N}_\beta$ be with $\crit(j)=\delta$. By way of contradiction, suppose $$\mathfrak{M}\models\text{$``\delta$ is not $C^{(n)}_{\mathbb{P}_\kappa}$-$\mathcal{S}$-extendible''}.$$ Thus $f(\delta)=\delta+\sigma$ for certain ordinal $\sigma$. Since $j\restriction \mathcal{M}_\alpha\colon \mathcal{M}_\alpha\rightarrow\mathcal{M}_\beta$ is an elementary embedding our previous comments yield $\delta\in C\cap \gamma_\alpha$. Also, $\gamma_\alpha$ is a closure point of $f$, so $\delta+\sigma=f(\delta)<\gamma_\alpha$. Thus there is an elementary embedding $j\colon V_{\delta+\sigma}\rightarrow V_{j(\delta+\sigma)}$ (namely, the restriction of $j$) with $\crit(j)=\delta$. 
    
    \smallskip
    Let us remark a few things: 
    \begin{enumerate}
        \item Since $\delta<j(\delta)$ and $j(\delta)\in C\cap \gamma_\beta$, 
        $$\delta+\sigma=f(\delta)<j(\delta).$$
        \item $j(\mathcal{S}\cap(\delta+\sigma))=j(\mathcal{S}\cap\gamma_\alpha)\cap j(\delta+\sigma)=(\mathcal{S}\cap \gamma_\beta)\cap j(\delta+\sigma)=\mathcal{S}\cap j(\delta+\sigma);$
        \item $\langle V_{j(\delta)},\in,\mathbb{P}_\kappa \cap V_{j(\delta)}\rangle\prec_{\Sigma_n}\langle V_\kappa,\in,\mathbb{P}_\kappa\rangle$ because $j(\delta)\in C.$
    \end{enumerate}
    
   From the above we infer that $\mathfrak{M}\models``\delta$ is $(\delta+\sigma)$-$C^{(n)}_{\mathbb{P}_\kappa}$-$\mathcal{S}$-extendible''  as witnessed by $j\colon \langle V_{\delta+\sigma},\in, \mathcal{S}\cap (\delta+\sigma)\rangle\rightarrow \langle V_{j(\delta+\sigma)},\in, \mathcal{S}\cap j(\delta+\sigma)\rangle$. Nonetheless this contradicts the fact that $f(\delta)=\delta+\sigma.$ 
\end{proof}

We are now in conditions to prove the promissed theorem:
\begin{theorem}\label{VopenkaIntheEnd}
 $V[G]_\kappa\models ``\kappa$ is Vop\v{e}nka''.
\end{theorem}
\begin{proof}
 This is equivalent to saying $V[G]_\kappa\models \mathrm{VP}$. To show  this we will employ Bagaria's characterization of $\mathrm{VP}$ (inside $V[G]_\kappa$); namely, $\mathrm{VP}$ holds if and only if, for each $n<\omega$, there is a $C^{(n)}$-extendible cardinal \cite{Bag}. 
 First, recall that $\mathcal{U}$ is a normal ultrafilter and that $\mathcal{S}\in\mathcal{U}$ so that $$A_0:=\{\delta<\kappa\mid \text{$\mathcal{S}\cap \delta$ is stationary}\}\in\mathcal{U}.$$
As argued in Lemmas~\ref{Nicesplitting} and \ref{Manyextendibles}, the following sets are also $\mathcal{U}$-large:
 $$A_1:=\{\delta<\kappa\mid \text{There is $\mathbb{D}^\delta\s \mathbb{P}_\kappa$ with $(\alpha)$ and $(\beta)$ in p.\pageref{ClausesDirectedClosure}}\},$$
 $$A_2:=\{\delta<\kappa\mid \langle V_\kappa,\in, \mathbb{P}_\kappa, \mathcal{S}\rangle\models\text{$``\delta$ is $C^{(n)}_{\mathbb{P}_\kappa}$-$\mathcal{S}$-extendible''}\}.$$

 We shall argue that each $\delta\in \cap_{i<3}A_i$ remains $C^{(n)}$-extendible in $V[G]_\kappa$. As in Theorem~\ref{Aholdsintheend} we denote $W=V_\kappa$ and note that $V[G]_\kappa=W[G]$.
 
 Fix a $\delta\in \cap_{i<3}A_i$ and let $\lambda>\delta$ be such that $\one\forces_{\mathbb{P}_\kappa} W[\dot{G}]_\lambda=W_\lambda[\dot{G}_\lambda]$. A cardinal $\lambda$ with the above-described property always exists (in fact there is a proper class of them below $\kappa$) by virtue of \cite[Lemma~5.7]{BP}.

 \begin{claim}
     Working in $W[G_\delta]$, the set defined by $$\mathcal{D}_\lambda:=\{r\in \mathbb{P}_\kappa/G_\delta\mid r\forces_{\mathbb{P}_\kappa/G_\delta} \varphi(\kappa,\delta,\lambda, n)\}$$ is dense. Here $\varphi(\kappa,\delta,\lambda, n)$ denotes the  assertion
$$\text{$``\exists\theta\in\kappa\,\exists \tau\colon W[G_\delta][\dot{G}_{[\delta,\kappa)}]_\lambda\rightarrow W[G_\delta][\dot{G}_{[\delta, \kappa)}]_\theta\,\crit(\tau)=\delta\;\text{and}\; \tau(\delta)\in \dot{C}^{(n)}\setminus (\lambda+1)$''.}$$
 \end{claim}
 \begin{proof}[Proof of claim]
 Let $r\in \mathbb{P}_\kappa/G_\delta$. Since $\mathbb{P}_\kappa$ has Easton support there is $\mu\in A_0$ inaccessible above $\lambda$ such that $r\in\mathbb{P}_\mu/G_\delta$. By our assumption that $\delta$ is $C^{(n)}_{\mathbb{P}_\kappa}$-$\mathcal{S}$-extendible in  $\langle W,\in,\mathbb{P}_\kappa,\mathcal{S}\rangle$ there is an elementary embedding $$j\colon \langle W_\mu,\in,\mathcal{S}\cap \mu\rangle\rightarrow \langle W_\theta,\in,\mathcal{S}\cap \theta\rangle$$ with $\theta<\kappa$, $\crit(j)=\delta$, $j(\delta)>\mu$ and $j(\delta)\in {C}^{(n)}_{\mathbb{P}_\kappa}$. 

 \smallskip

Since the tail $\mathbb{P}_{[\mu,\kappa)}$ is $\mu^+$-closed  the trivial condition of $\mathbb{P}_\mu$ forces $$\text{$``W[\dot{G}]_\lambda=W_\lambda[\dot{G}_\lambda]$''.}$$ Also,  $\mathbb{P}_\mu$ is a definable class in $\langle W_\mu, \in, \mathcal{S}\cap \mu\rangle$,  so that 
 $$\langle W_\mu, \in, \mathcal{S}\cap \mu\rangle\models \text{$``\one\forces _{\mathbb{P}_\mu}W_\lambda[\dot{G}_\lambda]=W[\dot{G}]_\lambda$''}$$
 and by elementarity, 
  $$\langle W_\theta, \in, \mathcal{S}\cap \theta\rangle\models \text{$``\one\forces _{\mathbb{P}_\theta}W_{j(\lambda)}[\dot{G}_{j(\lambda)}]=W[\dot{G}]_{j(\lambda)}$''.}$$

Since we have chosen $\delta$ in  $A_1$ there is a dense subiteration $\mathbb{D}^\delta$ of $\mathbb{P}_\kappa$ witnessing $(\alpha)$ and $(\beta)$ above. Since $\mathcal{S}\cap \mu$ is stationary  (as $\mu\in A_0$)  one has $$(\mathbb{D}^\delta)^{\langle W_\mu,\in, \mathcal{S}\cap \mu\rangle}=(\mathbb{D}^\delta)_\mu,$$ and,  by elementarity, 
$$(\mathbb{D}^{j(\delta)})^{\langle W_\theta,\in, \mathcal{S}\cap \theta\rangle}=(\mathbb{D}^{j(\delta)})_\theta.\footnote{The reader should bear in mind that by our prelimminary choice the well-ordering $\prec$ is absolute between $V_\mu$ (resp. $V_\theta$) and $V.$}$$
 Since $\delta\in A_1$ we have that $\one\forces_{\mathbb{D}_\delta}``\dot{\mathbb{D}}_{[\delta,\lambda)}$ is $\delta^+$-directed-closed'' and thus,
$$\text{$\one\forces_{\mathbb{D}_{j(\delta)}}``\dot{\mathbb{D}}_{[j(\delta),j(\lambda))}$ is $j(\delta)^+$-directed-closed''.}$$ 

Now work in $W[G_{j(\delta)}]$. Clearly $$\{j(p)\mid p\in G_{[\delta,\lambda)}\cap (\mathbb{D}^\delta)\}$$ is a directed subset of the tail forcing $$(\mathbb{D}^{j(\delta)}_{[j(\delta),j(\lambda))})$$ of cardinality ${<}j(\delta)$ so we can let $r^*\in \mathbb{D}^{j(\delta)}_{[j(\delta),j(\lambda))}\s \mathbb{P}_{j(\lambda)}/G_{j(\delta)}$ such that $$\text{$r^*\leq j(p)$ for all $p\in G_{[\delta,\lambda)}\cap (\mathbb{D}^\delta)_{[\delta,\lambda)}.$ }$$

Since $(\mathbb{D}^\delta)_\delta$ is an Easton-supported iteration it is easy to define  a condition $$r\wedge r^*\leq_{\mathbb{P}_{j(\lambda)}/G_\delta}r, r^*$$
(for details see \cite[p.319]{BP}).
This condition serves as a master condition and thus necessarily forces the existence of an elementary embedding 
$$\tau\colon  W_\mu[G_\delta][\dot{G}_{[\delta,\lambda)}] \rightarrow W_\theta[G_\delta][\dot{G}_{[\delta, j(\lambda))}].$$
Note that the restriction of $\tau$ to $W_\lambda[G_\lambda]$ is forced to be
$$\tau\restriction W_\lambda[\dot{G}_\lambda]\colon  W_\lambda[\dot{G}_\lambda] \rightarrow W_{j(\lambda)}[\dot{G}_{j(\lambda)}],$$
but by our preliminary comments this is equivalent to saying that 
$$\tau\restriction W_\lambda[\dot{G}_\lambda]\colon  W[\dot{G}]_\lambda \rightarrow W[\dot{G}]_{j(\lambda)}.$$
 On a different note, Fact~\ref{fact: preserving Cns} shows that $\one_{\mathbb{P}_\kappa}$ forces $``j(\delta)\in C^{(n)}$''.

 Putting everything into the same canopy one concludes that 
 $$W[G_\delta]\models r\wedge r^*\forces_{\mathbb{P}_\kappa/G_\delta}\varphi(\kappa,\delta,\lambda, n).$$
 This proves that, indeed, $\mathcal{D}_\lambda$ is a dense set in $W[G_\delta].$
 \end{proof}
Invoking the claim we let $r\in G_{[\delta,\kappa)}\cap \mathcal{D}_\lambda$, which yields an elementary embedding $j\colon W[G]_\lambda\rightarrow W[G]_\theta$ witnessing that $\delta$ is $\lambda$-$C^{(n)}$-extendible in $W[G]=V[G]_\kappa$. Thus we conclude that $\delta$ is $C^{(n)}$-extendible in $V[G]_\kappa$.
\end{proof}

Putting together Theorems~\ref{Aholdsintheend} and \ref{VopenkaIntheEnd} we infer that $V[G]_\kappa$ is a model of ZFC wittnessing the configuration claimed in Theorem~\ref{ConsistencyofA}.

\begin{remark}
Recall that  $``\mathrm{Ord}$ is Mahlo'' is the scheme asserting that for each first-order formula $\varphi(x,y)$, if for some parameter $a$, $\{\alpha\in\ord\mid \varphi(\alpha,a)\}$ is a closed unbounded class of ordinals then there is a regular cardinal $\theta$ for which $\varphi(\theta,a)$ holds. Note that $``\mathrm{Ord}$ is Mahlo'' holds in $V[G]_\kappa$. In fact, much more is true as Vop{e}nka's Principle holds in this model. Therefore in $V[G]_\kappa$ we obtain the stronger version of axiom $\mathcal{A}$ where a proper class of inaccessible cardinals is replaced by $``\mathrm{Ord}$ is Mahlo''.
\end{remark}

\subsection{$\mathcal{A}$ and  $I_0$ cardinals}\label{AandI0}
In this section we show that $\mathcal{A}$ is compatible with the existence of $I_0$-cardinals (see p.\pageref{DefiningI0} in \S\ref{Preliminaries} for definitions). Our departing hypothesis is that $\kappa$ carries a normal measure $\mathcal{U}$ such that, for each $X\s \kappa$, $$\{\delta<\kappa\mid \exists\lambda\, I_0(\delta,\lambda, X\cap \lambda)\}\in \mathcal{U}$$
and that $\prec$ is a well-ordering of $V$ witnessing equation \eqref{eq: dagger} in p.\pageref{eq: dagger}. Denote this assumption  $\oplus_\kappa$. Woodin has informed us that $\oplus_\kappa$ follows from the stronger hypothesis 
 $I_0^\sharp(\kappa,\lambda)$ employing Cramer's results on inverse reflection \cite{Cramer}. Recall that $I^\sharp_0(\kappa,\lambda)$ is a shorthand for ``There is a an elementary embedding with $j\colon L(V_{\lambda+1}^\sharp)\rightarrow L(V_{\lambda+1}^\sharp)$ with $\crit(j)=\kappa<\lambda$''. Another argument yielding $\oplus_\kappa$ can be found in \cite[Theorem~137]{WoodinSuitableII}. 



\begin{theorem}\label{MainTheoremwithStrongerlargecardinals}
  Assume that $\oplus_\kappa$ holds. Then there is a model for the theory $$\mathrm{ZFC}+\mathcal{A}+\exists\lambda\,I_0(\lambda).$$ 
\end{theorem}
Let $\mathcal{U}$ be a normal measure witnessing $\oplus_\kappa$. Arguing as in Lemma~\ref{indiscernibility} we find $\mathcal{S}\in\mathcal{U}$ consisting of indiscernibles such that $\{\delta<\kappa\mid \text{$\mathcal{S}\cap \delta$ is stationary}\}$ is $\mathcal{U}$-large. Thus, by virtue of $\oplus_\kappa$, the set  of all $\delta<\kappa$ such that
\begin{enumerate}
    \item[$(\alpha)$] $\mathcal{S}\cap \delta$ is stationary,
    \item[$(\beta)$] and $I_0(\delta,\lambda,\mathcal{S}\cap \lambda)$ holds for some $\lambda$\label{ClausesReflectingI0}
\end{enumerate}
is $\mathcal{U}$-large as well. Let $\delta$ be witnessing $(\alpha)$ and $(\beta)$ above and let us show that it witnesses $I_0(\delta,\lambda)$ in certain generic extension by $\mathrm{Add}(\omega,1)\ast \dot{\mathbb{P}}_\kappa$. To simplify notations our ground model $V$ will be a generic extension by Cohen forcing $\mathrm{Add}(\omega,1)$. The preservation of $I_0(\delta,\lambda)$ (as in Theorem~\ref{VopenkaIntheEnd}) requires passing to a dense subiteration of $\mathbb{P}_\kappa$ which is nicely behaved.

\smallskip

Let $j\colon L(V_{\lambda+1})\rightarrow L(V_{\lambda+1})$ be a witness for $I_0(\delta,\lambda, \mathcal{S}\cap \lambda)$ and denote by $\langle \delta_n\mid n<\omega\rangle$ the corresponding critical sequence. Since $\mathcal{S}\cap\delta$ is stationary and $j(\mathcal{S}\cap\lambda)=\mathcal{S}\cap\lambda$ it follows that
$\mathcal{S}\cap\delta_n$ is stationary for all $n<\omega.$

Define an Easton-supported iteration $\mathbb{D}_\kappa:=\mathbb{D}(\kappa, \mathcal{S},\langle \delta_n\mid n<\omega\rangle)$ as follows. Suppose that $\mathbb{D}_\varepsilon$ has been defined and let $G_\varepsilon\s \mathbb{D}_\varepsilon$ a $V$-generic. 
If $V[G_\varepsilon]\models``\varepsilon$ is not inaccessible'', $(\dot{\mathbb{Q}}^*_\varepsilon)_{G_\varepsilon}$ is declared to be the trivial poset. Otherwise we define it according to the following casuistics:

\begin{definition}\label{DefinitionDenseIteration}$(\dot{\mathbb{Q}}^*_\varepsilon)_{G_\varepsilon}$ is defined according to the following casuistic:
\smallskip

{\textbf{Case $(\aleph)$:}} If $\mathcal{S}\cap\varepsilon$ is non-stationary in $V$ then 
$$(\dot{\mathbb{Q}}^*_\varepsilon)_{G_\varepsilon}:=\begin{cases}
   (\dot{\mathbb{Q}}_\varepsilon)_{G_\varepsilon}, & \text{if $\varepsilon<\delta_0$;}\\
    (\dot{\mathbb{D}}_{\varepsilon,\delta_n})_{G_\varepsilon}, & \text{if $\varepsilon>\delta_0$ and $n=\max\{m<\omega\mid \delta_m<\varepsilon\}$.} 
\end{cases}$$

\smallskip

{\textbf{Case $(\beth)$:}}\label{CaseBeth} If $\mathcal{S}\cap\varepsilon$ is stationary in $V$ then
\label{Casealeph}
$(\dot{\mathbb{Q}}^*_\varepsilon)_{G_\varepsilon}:=(\dot{\mathbb{Q}}_\varepsilon)_{G_\varepsilon}.$
\end{definition}

Here  $(\dot{\mathbb{Q}}_\varepsilon)_{G_\varepsilon}$ is as in Definition~\ref{DefinitionMainIteration} and  $ (\dot{\mathbb{D}}_{\varepsilon,\delta_n})_{G_\varepsilon}$  is as in Claim~\ref{ClaimDefinableDirected}.

Finally, define
$$\mathbb{D}_\kappa:=\varinjlim\langle \mathbb{D}_\varepsilon;\dot{\mathbb{Q}}^*_\varepsilon\mid \varepsilon<\kappa\rangle.$$

A few datapoints in regards to $\mathbb{D}_\kappa$:
\begin{enumerate}
    \item $\mathbb{D}_\kappa$ (resp. $\mathbb{D}_\lambda$) is a dense subiteration of $\mathbb{P}_\kappa$ (resp. $\mathbb{P}_\lambda$);
    \item $\mathbb{D}_\lambda$ is definable in $L(V_{\lambda+1})$ via $\lambda, \mathcal{S}\cap \lambda,\prec$ and $\langle \delta_n\mid n<\omega\rangle$;
    \item \textbf{Case~($\aleph$)} above is well-posed as it precludes $\varepsilon$ be  one of the $\delta_n$'s.
    \item For each $n<\omega$, $\one\forces_{\mathbb{D}_{\delta_n}}\dot{\mathbb{D}}_{[\delta_n,\delta_{n+1})}$ is $\delta_n^+$-directed-closed.\label{DirectedClosureOfDenseIteration}
\end{enumerate}
We remind our readers that  $\mathbb{P}_\kappa$ (resp. $\mathbb{P}_\lambda$) was the ($\lambda$-stage of) the main iteration of \S\ref{SectionForcingW}; namely, the iteration given in Definition~\ref{DefinitionMainIteration} using $\mathcal{S}\cap\lambda.$
\begin{lemma}\label{LemmapreservingI0}
   $I_0(\delta,\lambda)$ holds in certain generic extension by $\mathbb{P}_\kappa$. 
\end{lemma}
\begin{proof}
First note that $j(\mathbb{D}_\lambda)=\mathbb{D}(\lambda, \mathcal{S}\cap\lambda,\langle \delta_n\mid n\geq 1\rangle)$. That is, $j(\mathbb{D}_\lambda)$ is an iteration which agrees with $\mathbb{P}_\lambda$ up to $(\delta_1+1)$ and beyond that it agrees with $\mathbb{D}_\lambda$. Clearly, $\mathbb{D}_\lambda\s j(\mathbb{D}_\lambda)$ and both  are dense subiterations of $\mathbb{P}_\lambda$.
\begin{claim}\label{TheDsaresmall}
    For each $n<\omega$, $|\mathbb{D}_\gamma|<\delta_n$ for all $\gamma<\delta_n$.
\end{claim}
\begin{proof}[Proof of claim]
    Suppose towards a contradiction that $|\mathbb{D}_{{\gamma_*}}|\geq \delta_0$ for some ${\gamma_*}<\delta_0$. By elementarity and since $j({\gamma_*})={\gamma_*}$, 
$$|j(\mathbb{D}_{\gamma_*})|=|j(\mathbb{D}_\lambda)_{j({\gamma_*})}|=|j(\mathbb{D}_\lambda)_{\gamma_*}|=|\mathbb{P}_{\gamma_*}|\geq \delta_1.$$
Since $j(\mathbb{P}_{\gamma_*})=\mathbb{P}_{\gamma_*}$ (because $\mathbb{P}_\lambda$ is definable via $\mathcal{S}\cap\lambda$) one can repeat this $\omega$-many times and conclude that $|\mathbb{P}_{\gamma_*}|\geq \lambda$, which is certainly false.

\smallskip

Fix an arbitrary $1\leq n<\omega$ and let us show that $|\mathbb{D}_\gamma|<\delta_{n}$ for all $\gamma<\delta_{n}$. By elementarity and our base case, $|j^{n}(\mathbb{D}_\lambda)_\gamma|<\delta_{n}$ for all $\gamma<\delta_{n}$. 

For each $\gamma<\delta_{n}$ we have  $$j^{n}(\mathbb{D}_\lambda)_\gamma=\mathbb{D}(\lambda,\mathcal{S}\cap\lambda, \langle \delta_{m}\mid m\geq n\rangle)_\gamma=\mathbb{P}_\gamma\supseteq \mathbb{D}_\gamma.$$
Therefore clearly $|\mathbb{D}_\gamma|<\delta_{n}$ for all $\gamma<\delta_{n}$. This completes the claim.
\end{proof}
The next is the master condition argument needed to lift our embedding:
\begin{claim}\label{ClaimLifting}
    There is $q\in j(\mathbb{D}_\lambda)$ with $q\forces_{j(\mathbb{D}_\lambda)}``p\in\dot{G}\cap\mathbb{D}_\lambda\,\Rightarrow\, j(p)\in\dot{G}$''.
\end{claim}
\begin{proof}[Proof of claim]
    The construction of $q$ is divided in three cases. 

\smallskip

$\br_{(\mathrm{I})}$ Let $q\restriction\delta_1$ be the trivial condition in $j(\mathbb{D}_\lambda)_{\delta_1}$.

\smallskip

$\br_{(\mathrm{II})}$ For the construction of the $j(\mathbb{D}_\lambda)_{\delta_n}$-name $q(\delta_n)$ notice that
\begin{enumerate}
    \item $\one\forces_{j(\mathbb{D}_\lambda)_{\delta_n}}\{j(p)(\delta_n)\mid p\in \dot{G}\cap\mathbb{D}_\lambda\}$ is a directed subset of  size ${<\delta_n}$, 
    \item and $\one\forces_{j(\mathbb{D}_\lambda)_{\delta_n}}\dot{\mathbb{Q}}^*_{\delta_n}$ is $\delta_n^+$-directed-closed.
\end{enumerate}

For Clause~(1): The trivial condition of $j(\mathbb{D}_\lambda)_{\delta_n}$ forces  $$\{j(p)(\delta_n)\mid p\in \dot{G}\cap \mathbb{D}_\lambda\}\s\{j(p(\delta_{n-1}))\mid (p\restriction \delta_{n-1}+1)\in \dot{G}\cap\mathbb{D}_{\delta_{n-1}+1}\}$$
and by the previous claim $|\mathbb{D}_{\delta_{n-1}+1}|<\delta_n$.

\smallskip

For Clause~(2): 
Since $\mathcal{S}\cap\delta_n$ is stationary  when defining $\mathbb{Q}^*_{\delta_n}$ we fall in \textbf{Case $(\beth)$} of Definition~\ref{DefinitionDenseIteration}, so the poset is $\delta_n^+$-directed-closed.

Finally, combine (1) and (2)  to produce a $j(\mathbb{D}_\lambda)_{\delta_n}$-name $q(\delta_n)$ such that $$\one\forces_{\mathbb{P}_{\delta_n}}\forall p\in \dot{G}\cap\mathbb{D}_\lambda\, (q(\delta_n)\leq j(p)(\delta_n)).$$

$\br_{(\mathrm{III})}$ For the construction of the $j(\mathbb{D}_\lambda)_{\delta_n+1}$-name $q\restriction (\delta_n,\delta_{n+1})$ one argues similarly. There are two key observations that one has to make; namely,
\begin{enumerate}
    \item $\one\forces_{j(\mathbb{D}_\lambda)_{\delta_n+1}}|(j``\mathbb{D}_\lambda)\restriction (\delta_n,\delta_{n+1})|\leq\delta_n$;
    \item $\one \forces_{j(\mathbb{D}_\lambda)_{\delta_n+1}}j(\mathbb{D}_\lambda)\restriction(\delta_n,\delta_{n+1})$ is $\delta_n^+$-directed-closed.
\end{enumerate}
Once the above is established it is easy to find a master condition; to wit, a $j(\mathbb{D}_\lambda)_{\delta_n+1}$-name $q\restriction (\delta_n,\delta_{n+1})$ for a condition in $j(\mathbb{D}_\lambda)\restriction(\delta_n,\delta_{n+1})$ such that $$\one\forces_{j(\mathbb{D}_\lambda)_{\delta_n+1}} \forall p\in \dot{G}\cap\mathbb{D}_\lambda\, (q\restriction (\delta_n,\delta_{n+1})\leq j(p)\restriction (\delta_n,\delta_{n+1})).$$
By design of  $j(\mathbb{D}_\lambda)$  the tail forcing $j(\mathbb{D}_\lambda)\restriction (\delta_n,\delta_{n+1})$ is $\delta_n^+$-directed-closed (see Clause~(4) in p.\pageref{DirectedClosureOfDenseIteration}). For this we have crucially used that $n\geq 1$ as then 
$$\one\forces_{j(\mathbb{D}_\lambda)_{\delta_n+1}}j(\mathbb{D}_\lambda)\restriction (\delta_n,\delta_{n+1})=\mathbb{D}_\lambda\restriction (\delta_n,\delta_{n+1}).$$


\smallskip

Suppose towards a contradiction that (1) is false and let $H\s j(\mathbb{D}_\lambda)_{\delta_n+1}$ a generic filter such that in $L(V_{\lambda+1})[H]$ there is an injective sequence $$\langle j(p_\alpha)\restriction(\delta_n,\delta_{n+1})\mid \alpha<\delta_n^+\rangle.$$
Since Easton-supports are taken at stage $\delta_n$, there is $\sigma_\alpha<\delta_n$ such that $\supp(p_\alpha\restriction\delta_n)\s \sigma_\alpha$. Working in $L(V_{\lambda+1})$ let $I\s \delta_n^+$ be with $|I|=\delta_n^+$ such that $\alpha\mapsto\sigma_\alpha$ is constant, say with value $\sigma_*$. Thus, in $L(V_{\lambda+1})[H]$, 
$$\langle j(p_\alpha)\restriction j(\sigma_*)\mid \alpha\in I\rangle$$
is still injective. By elementarity, this implies that
$$\langle p_\alpha\restriction \sigma_*\mid \alpha\in I\rangle$$
is also injective and thus $|\mathbb{D}_{\sigma_*}|\geq \delta_n^+$. Note that this latter assertion is precluded by Claim~\ref{TheDsaresmall}, so we obtain the desired contradiction.

\medskip

Let us check that $q$ forces $``p\in \dot{G}\cap\mathbb{D}_\lambda\Rightarrow j(p)\in \dot{G}$''. Let $q\in G$ and $p\in G\cap \mathbb{D}_\lambda$. Then $j(p)\restriction\delta_1\geq   p\restriction \delta_1\in G_{\delta_1}$. By construction, for each $n\geq 1$, $q\restriction [\delta_n,\delta_{n+1}]\leq j(p)\restriction [\delta_n,\delta_{n+1}]$, hence the latter belongs to $G_{[\delta_n,\delta_{n+1}]}$.
\end{proof}
We are now in conditions to lift $j\colon L(V_{\lambda+1})\rightarrow L(V_{\lambda+1})$. Let $G\s\mathbb{P}_\kappa$ be a $V$-generic with $q\in G$ (recall that $\mathbb{P}_\kappa$ is the iteration of \S\ref{SectionForcingW}). Since both $\mathbb{D}_\lambda$ and $j(\mathbb{D}_\lambda)$ are dense in $\mathbb{P}_\lambda$ we have that $G_\lambda\cap \mathbb{D}_\lambda$ and $G_\lambda\cap j(\mathbb{D}_\lambda)$ are $L(V_{\lambda+1})$-generics for the corresponding posets. 

\smallskip

By the previous claim it follows that $j$ lifts to
$$j\colon L(V_{\lambda+1})[G_\lambda\cap \mathbb{D}_\lambda]\rightarrow L(V_{\lambda+1})[G_\lambda\cap j(\mathbb{D}_\lambda)].$$
By density of $\mathbb{D}_\lambda$ in $\mathbb{P}_\lambda$, $$L(V_{\lambda+1})[G_\lambda\cap \mathbb{D}_\lambda]=L(V_{\lambda+1})[G_\lambda]$$ and by $\aleph_1$-closure of $\mathbb{P}_\lambda$,
$$L(V_{\lambda+1})[G_\lambda]=L(V[G_\lambda]_{\lambda+1}).$$
Similar comments apply to $L(V_{\lambda+1})[G_\lambda\cap j(\mathbb{D}_\lambda)]$. Thereby, $j$ lifts to 
$$j\colon L(V[G_\lambda]_{\lambda+1})\rightarrow L(V[G_\lambda]_{\lambda+1}).$$
Since $\lambda$ is singular, the iteration $\mathbb{P}_\kappa$ factors as a three-step iteration $$\mathbb{P}_\lambda\ast \{\one\}\ast \mathbb{P}_{(\lambda,\kappa)}.$$ Here the latter forcing is forced to be more than $(2^{\beth_\lambda})^+_{V[G_\lambda]}$-closed so that  $V[G]_{\lambda+1}=V[G_{\lambda}]_{\lambda+1}$. All in all this finally yields
\begin{equation*}
    j\colon L(V[G]_{\lambda+1})\rightarrow L(V[G]_{\lambda+1}).\qedhere
\end{equation*}
\end{proof}
Combining Lemma~\ref{LemmapreservingI0} with Theorem~\ref{Aholdsintheend} we get  Theorem~\ref{MainTheoremwithStrongerlargecardinals}.

\begin{remark}
    We do not know if the above argument can be adapted to get $\exists\lambda\, I_0(\delta,\lambda)$ for all $\delta$ witnessing Clauses~$(\alpha)$ and $(\beta)$ in p.\pageref{ClausesReflectingI0}. The issue is that for each such $\delta$ we have to choose a generic $G\s \mathbb{P}_\kappa$ which contains a condition $q$ witnessing  Claim~\ref{ClaimLifting} -- this condition depends upon $\delta$. If Claim~\ref{ClaimLifting} were true for a dense set of $q$'s then we would be able to prove that for all $\delta$'s witnessing $(\alpha)$ and $(\beta)$ $\exists\lambda\, I_0(\delta,\lambda)$ holds: Indeed, let $G\s \mathbb{P}_\kappa$ a $V$-generic and $\delta$ as above. Let $j\colon L(V_{\lambda+1})\rightarrow L(V_{\lambda+1})$ and define $\mathbb{D}_\lambda:=\mathbb{D}(\lambda,\mathcal{S}\cap\lambda, \langle \delta_n\mid n<\omega\rangle)$. Since Claim~\ref{ClaimLifting} works for a dense collection of $q\in j(\mathbb{D}_\lambda)$ then there is $q\in G_\lambda$ with that property. Using this we can argue as exactly before that $j$ lifts to $  j\colon L(V[G]_{\lambda+1})\rightarrow L(V[G]_{\lambda+1}).$
\end{remark}

\subsection{Open questions}\label{SectionOpenquestionsaboutW}
Let us consider an interesting strengthening of $\mathcal{A}$: 
\begin{definition}
Axiom $\mathcal{A}^{+}$ is the conjunction of the following assertions:
    \begin{enumerate}
\item The scheme $``\mathrm{Ord}$ is Mahlo'' holds;
 \item Suppose that $\delta$ is a stationary-correct superstrong with inaccessible target $\lambda$. Then, $\delta$ is tall with target $\theta$ \textbf{for all}  inaccessible  $\theta\in (\delta,\lambda)$.
 \end{enumerate}
\end{definition}

 An immediate consequence of $\mathcal{A}^{+}$ is the equivalence between supercompactness and $C^{(n)}$-supercompactness for all complexities $n<\omega$. So  $$\text{$\mathcal{A}^{+}+``$There is a supercompact''}$$ is incompatible with Woodin's EEA (see Definition~\ref{EEA})  in a very strong sense. Recall that so did  $\mathcal{A}+``$There is a supercompact'' (Theorem~\ref{AprecludesEEA}).
 \begin{prop}\label{pro: axiomA+}
     Assume $\mathcal{A}^{+}$ holds. Let $\delta$ be supercompact, $\lambda>\delta$ and  $n<\omega$. Then, 
     $\text{$\delta$ is $\lambda$-supercompact if and only if   $\delta$ is $\lambda$-$C^{(n)}$-supercompact.}$
 \end{prop}
 \begin{proof}
     By Theorem~\ref{CharacterizingCnsupercompacts} it suffices to show that, for each $n<\omega$, $\delta$ is an enhanced $C^{(n)}$-tall cardinal.  Fix $\delta<\lambda$  and let $\theta\in C^{(n)}$ be inaccessible above $\lambda$ -- this is possible as $\ord$ is Mahlo. Let $j\colon V\rightarrow M$ be a $\theta^+$-supercompact embedding with $\crit(j)=\delta$. Since $M$ thinks that $\delta$ is $(\delta+1)$-extendible and $\theta$ is inaccessible, axiom $\mathcal{A}^{+}$ yields a $(\delta,\theta)$-tall extender $E\in M$ such that $j_E(\delta)=\theta$.  In particular $\delta$ is tall with target $\theta$ in $V$, as sought.
 \end{proof}
The consistency of $\mathcal{A}^{+}$ with very large cardinals is  completely open:
\begin{question}\label{QuestionA+}
    Is $\mathcal{A}^{+}$ consistent with a proper class of extendible cardinals?
\end{question}
Our belief is that the consistency of $\mathcal{A}^{+}$  with (even) a proper class of supercompacts (if possible at all) will require completely new methods. In fact, the consistency of the following weaker configuration is widely open:
\begin{question}
   Is it consistent to have a proper class of supercompact cardinals and that every supercompact cardinal is $C^{(2)}$-supercompact?
\end{question}
An intriguing inner-model-theoretic question is how axiom $\mathcal{A}$ fits with  \emph{Woodin's $\mathrm{Ultimate-}L$ program}. 
Woodin  conjectured (2018) that  EEA is implied by the axiom $V=\mathrm{Ultimate-}L$  and in Theorem~\ref{AprecludesEEA} we showed that $\mathcal{A}$ is incompatible with EEA.\footnote{For the definition of the axiom $V=\mathrm{Ultimate-}L$ see \cite{midrasha}.} Nonetheless we ask: 
\begin{question}
    Does $V=\mathrm{Ultimate-}L$ imply $\mathrm{EEA}$? 
\end{question}
A question of a similar flavor is whether under $V=\mathrm{Ultimate-}L$ every tall cardinal with target past a regular cardinal $\gamma$ is in fact $\gamma$-strong. 
\begin{question}
    Does $V=\mathrm{Ultimate-}L$ yield the equivalence between tall and strong cardinals proved in  \cite{FernandesSchidler}? In general, is the configuration obtained in \cite{FernandesSchidler} compatible with a proper class of supercompact cardinals?
\end{question}
\subsection{Axiom $\mathcal{A}$ and the Ultimate-$L$ conjecture}\label{sec: ultimate L}
We would like to close this section devoting a few words to the connections between axiom $\mathcal{A}$ (resp. axiom  $\mathcal{A}^+$) and \emph{Woodin's $\mathrm{Ultimate-}L$ Conjecture}. Specifically, we would like to show that $\mathcal{A}$ imposes some limitations on the extent to which $\mathrm{Ultimate-}L$ resembles the set-theoretic universe.
\begin{conjecture}[Woodin, The $\mathrm{Ultimate-}L$ conjecture]
    The theory $$\text{$\mathrm{ZFC}+\text{``$\delta$ is an extendible cardinal''}$}$$ proves the existence of an inner model $N$ with:
    \begin{enumerate}
        \item $N$ has the $\delta$-cover, $\delta$-approximation and $\delta$-genericity properties;
        \item $N\models ``V=\mathrm{Ultimate-}L$''.
    \end{enumerate}
\end{conjecture}
Let us assume that the following  conjectures made by Woodin are true: 
\begin{enumerate}
    \item[$(\mathcal{C}_1)$] ZFC+$V=\mathrm{Ultimate-}L\vdash\mathrm{EEA}$;
    \item[$(\mathcal{C}_2)$] The $\mathrm{Ultimate-}L$ conjecture.
\end{enumerate}
Denote by $V$ the model for $\ZFC+\mathcal{A}+\mathrm{VP}$ obtained in Theorem~\ref{ConsistencyofA}. Combining $(\mathcal{C}_1)$ and $(\mathcal{C}_1)$ we get an inner model $N\s V$ with the $\delta$-covering, $\delta$-approximation and $\delta$-genericity properties satisfying $\mathrm{EEA}$. The combination of these three  implies that $N$ satisfies  \emph{Woodin's Universality Theorem} \cite{midrasha}. Thus, $N$ inherits all the ($C^{(1)}$-)supercompact/extendible cardinals past $\delta$. 
Let $\kappa$ be, for instance, the first supercompact cardinal past $\delta$. By Theorem~\ref{AprecludesEEA}, $\mathcal{A}$ implies that $\kappa$ is $C^{(1)}$-supercompact and the same holds in $N$ by Woodin's Universality theorem.  Since EEA holds in $N$, $\kappa$ is almost-$C^{(1)}$-extendible  (Theorem~\ref{AprecludesEEA}), hence a $\Sigma_3$-correct cardinal in $N$. But $N$ inherits from $V$ a proper class of supercompact cardinals so one concludes that $N\models ``\kappa$ is a supercompact limit of supercompact cardinals''. 

\smallskip

The above shows that even under the \emph{Ultimate-L Conjecture}, there may still be significant discrepancies between \(\mathrm{Ultimate}-L\) and \(V\). Similar (and indeed more radical) discrepancies arise if axiom \(\mathcal{A}^+\) was compatible with a proper class of extendible cardinals. While this does not endanger Woodin's conjectures, it highlights some limitations regarding how \(\mathrm{Ultimate-}L\) resembles \emph{\(V\) under the \(\mathrm{Ultimate-}L\) conjecture}. We believe these limitations are interesting from both mathematical and philosophical perspectives.

\section{On a question of Gitman and Goldberg}\label{GitmanGoldbergQ}
In private communication \cite{GG}, V. Gitman and G. Goldberg presented us with cardinal-preserving extendible cardinals. Just recently we learnt that this notion is due to V. Gitman and J. Osinski:
\begin{definition}[Gitman and Osinski]
    A cardinal $\kappa$ is called \emph{$\lambda$-cardinal-preserving extendible}  if 
there is an elementary embedding $j\colon V_\lambda\rightarrow M$ with $\crit(j)=\kappa$,  $j(\kappa)>\lambda$ and $\mathrm{Card}^M=\mathrm{Card}\cap M.$ Similarly, $\kappa$ is \emph{cardinal-preserving extendible} if it is {$\lambda$-cardinal-preserving extendible} for all $\lambda>\kappa$.
\end{definition}
This is a \emph{rank-version} of cardinal-preserving elementary embeddings studied by Caicedo and Woodin in \cite{Caicedo}. Recently, Goldberg showed that every cardinal-preserving extendible must be strongly compact \cite{GoldbergCardinalPresenving}. In private communication, Gitman and Goldberg asked the author whether every supercompact cardinal must be cardinal-preserving extendible. 

\begin{theorem}\label{TheoremGG}
   Suppose that $\kappa$ is a supercompact cardinal. Then there is a generic extension where  $\kappa$ is the first supercompact but it is not cardinal-preserving extendible.
\end{theorem}
\begin{proof}
The proof employs Radin forcing. We use the approach and notations of \cite{CumWoo}. Our readers may find a more concise account in  
\cite[\S3.1.1]{PovPhD}.

    Let $u\in\mathcal{U}_\infty$ be a measure sequence whose corresponding Radin forcing $\mathbb{R}_u$ preserves supercompactness of $\kappa$. This result is due to Cummings and Woodin \cite{CumWoo} (see also \cite[Corollary~3.1.21]{PovPhD}). 

    \smallskip

    Set $$A:=\{v\in\mathcal{U}_\infty\cap V_\kappa\mid \text{$\kappa_v$ is measurable in $V$}\}.$$ This is a $\mathcal{F}(u)$-large set. Forcing with $\mathbb{R}_u$ below  $\langle (u,A)\rangle$ yields a generic club $C\s \kappa$ consisting of $V$-measurable cardinals. By virtue of a theorem of D\v{z}amonja and Shelah \cite{DzjaSh}, every cardinal $\lambda\in C$ of countable $V^{\mathbb{R}_u}$-cofinality carries a $\square^*_\lambda$-sequence an the existence of  $\square^*_\lambda$-sequence is absolute between inner models agreeing on $\lambda^+$. For the definition of the $\square^*_\lambda$-principle the reader may want to consult \cite[Definition~2.3]{CumForMag}. 

    \smallskip

  We claim that $$\text{$\langle (u,A)\rangle\forces_{\mathbb{R}_u}``\kappa$ is not $(\kappa+1)$-cardinal-preserving extendible''.}$$ Suppose otherwise and let $G\s \mathbb{R}_u$ a $V$-generic filter with $\langle (u,A)\rangle\in G$ such that $\kappa$ is $(\kappa+1)$-cardinal-preserving extendible in $V[G].$ Note that $\kappa$ is also supercompact in $V[G]$ thanks to our choice of the measure sequence.
  
  Let $j\colon V[G]_{\kappa+1}\rightarrow M$ be an embedding witnessing that $\kappa$ is $(\kappa+1)$-cardinal-preserving extendible. 
 Let $\lambda\in j(C)\cap (E^{j(\kappa)}_\omega)^{V[G]}$ be an $M$-cardinal (hence a $V[G]$-cardinal) greater than $\kappa$. By elementarity, $M$ thinks that $\square^*_\lambda$ holds and since $(\lambda^+)^M=(\lambda^+)^{V[G]}$,  $\square^*_\lambda$ holds in $V[G]$. Since $\cf^M(\lambda)=\omega$, $\cf^{V[G]}(\lambda)=\omega$. Combining these two facts it follows that $\square^*_\lambda$ holds at a singular cardinal of cofinality ${<}\kappa$, which contradicts  the supercompactness of $\kappa$ by a theorem of Shelah (see \cite[Fact~2.5]{CumForMag}). 
\end{proof}
Bearing in mind that every extendible cardinal is cardinal-preserving extendible and that every cardinal-preserving extendible is strongly compact the above yields the next corollary:
\begin{cor}
    It is consistent for the first cardinal-preserving extendible be greater than the first supercompact.
\end{cor}
Yet another related question is whether cardinal-preserving extendible are necessarily extendible cardinals. After a visit to Harvard in the Spring 2024, J. Osinski posed the author this very question. The author showed that if $\delta$ is extendible and $\hod$ is cardinal correct then $\delta$ is cardinal-correct extendible in $\hod$. Combining this with unpublished results with Goldberg, one may show that the first cardinal-correct extendible can be the first strongly compact cardinal. This theorem will appear as a joint work in Osinski's forthcoming Ph.D. dissertation \cite{Osinski}.

\smallskip

  For a class $\mathcal{C}\s \mathrm{Card}$ let us say that a cardinal $\kappa$ is \emph{tall with targets in $\mathcal{C}$} if for each $\lambda>\kappa$ there is an elementary embedding $j\colon V\rightarrow M$ such that $\crit(j)=\kappa$, $j(\kappa)>\lambda$, $M^\kappa\s M$ and $j(\kappa)\in\mathcal{C}$. The next spin off of Theorem~\ref{TheoremGG} yields a prototype model where $\mathcal{A}$ fails:
\begin{theorem}\label{ConsistencyofAdoesnotHold}
	Assume the $\mathrm{GCH}$ holds and that there is a proper class of inaccessible cardinals. In the Radin-like extension of Theorem~\ref{TheoremGG}, $\kappa$ is not a $\mathrm{Lim}$-tall cardinal and thus axiom $\mathcal{A}$ fails.
\end{theorem}
\begin{proof}
	We follow the notations of Theorem~\ref{TheoremGG}. The claim is that $$\text{$\langle (u,A)\rangle\forces_{\mathbb{R}_u}``\kappa$ is not $C^{(1)}$-tall''.}$$
	Suppose otherwise and let $j\colon V[G]\rightarrow M$ be an elementary embedding with $\crit(j)=\kappa$, $j(\kappa)\in \mathrm{Lim}$ and $M^\kappa\s M$. Since $j(\kappa)$ is a limit cardinal in $V[G]$ there is a $V[G]$-cardinal $\lambda\in j(C)$ of countable $M$-cofinality (hence of countable $V[G]$-cofinality). Now, since the GCH holds and $M^\omega\s M$, $$(\lambda^+)^M=(\lambda^{\aleph_0})^M=(\lambda^{\aleph_0})^{V[G]}=(\lambda^+)^{V[G]}.$$ The contradiction is now achieved  as in the previous theorem. 
\end{proof}
The above suggests that all the embeddings witnessing supercompactness of $\kappa$ in the Radin-like extension arise from measures on $\mathcal{P}_\kappa(\lambda)$.  Succinctly putted, it suggests that supercompactness is the strongest large-cardinal notion preserved by Radin forcing. 
This yields two interesting problems:
\begin{question}
    Is there a Radin-like extension where $\kappa$ is both supercompact and tall with targets in  $\mathrm{Succ}$?
\end{question}

\begin{question}
What kind of tall embeddings may exist in a Radin extension?
\end{question}

\end{document}